\documentclass[10pt,reqno]{amsart}

\usepackage{amsmath,amsfonts,amssymb,mathabx}
\usepackage{graphicx}
\usepackage{caption}
\usepackage{subcaption}
\usepackage{comment}
\usepackage{proof}
\usepackage{cancel}
\usepackage[T1]{fontenc}
\usepackage{hyperref,cite}
\usepackage{multicol}
\usepackage[active]{srcltx}
\usepackage{mathrsfs}
\usepackage{cleveref}
\usepackage{float}

\graphicspath{{./images/}}
\newcommand{\minimatrix}[4]{\begin{bmatrix}#1&#2\\#3&#4\end{bmatrix}}

\renewcommand{\phi}{\varphi}
\renewcommand{\Re}[1]{\operatorname{Re} #1 }
\renewcommand{\Im}[1]{\operatorname{Im} #1}

\newcommand{\Aut}{\operatorname{Aut}}

\newcommand{\Z}{\mathbb{Z}}
\newcommand{\R}{\mathbb{R}}
\newcommand{\Q}{\mathbb{Q}}
\newcommand{\C}{\mathbb{C}}

\newcommand{\T}{\mathbb{T}}

\newcommand{\im}{\operatorname{im}}
\newcommand{\Irr}{\operatorname{Irr}}

\renewcommand{\vec}[1]{{\bf #1}}

\newtheorem{thm}{Theorem}[section]
\newtheorem{lem}[thm]{Lemma}
\newtheorem{prop}[thm]{Proposition}

\theoremstyle{definition}

\newtheorem{example}[thm]{Example}

\let\oldenumerate=\enumerate
	\def\enumerate{
	\oldenumerate
	\setlength{\itemsep}{3pt}
	}
\let\olditemize=\itemize
	\def\itemize{
	\olditemize
	\setlength{\itemsep}{3pt}
	}

\makeatletter
\def\legendre@dash#1#2{\hb@xt@#1{%
  \kern-#2\p@
  \cleaders\hbox{\kern.5\p@
    \vrule\@height.2\p@\@depth.2\p@\@width\p@
    \kern.5\p@}\hfil
  \kern-#2\p@
  }}
\def\@legendre#1#2#3#4#5{\mathopen{}\left(
  \sbox\z@{$\genfrac{}{}{0pt}{#1}{#3#4}{#3#5}$}%
  \dimen@=\wd\z@
  \kern-\p@\vcenter{\box0}\kern-\dimen@\vcenter{\legendre@dash\dimen@{#2}}\kern-\p@
  \right)\mathclose{}}
\newcommand\legendre[2]{\mathchoice
  {\@legendre{0}{1}{}{#1}{#2}}
  {\@legendre{1}{.5}{\vphantom{1}}{#1}{#2}}
  {\@legendre{2}{0}{\vphantom{1}}{#1}{#2}}
  {\@legendre{3}{0}{\vphantom{1}}{#1}{#2}}
}
\makeatother

\makeatletter
\def\legendre@dash#1#2{\hb@xt@#1{%
  \kern-#2\p@
  \cleaders\hbox{\kern.5\p@
    \vrule\@height.2\p@\@depth.2\p@\@width\p@
    \kern.5\p@}\hfil
  \kern-#2\p@
  }}
\def\@legendret#1#2#3#4#5{\mathopen{}\bigg(
  \sbox\z@{$\genfrac{}{}{0pt}{#1}{#3#4}{#3#5}$}%
  \dimen@=\wd\z@
  \kern-\p@\vcenter{\box0}\kern-\dimen@\vcenter{\legendre@dash\dimen@{#2}}\kern-\p@
  \bigg)\mathclose{}}
\newcommand\legendret[2]{\mathchoice
  {\@legendret{0}{1}{}{#1}{#2}}
  {\@legendret{1}{.5}{\vphantom{1}}{#1}{#2}}
  {\@legendret{2}{0}{\vphantom{1}}{#1}{#2}}
  {\@legendret{3}{0}{\vphantom{1}}{#1}{#2}}
}
\makeatother

\allowdisplaybreaks

\begin{document}

 \title[The graphic nature of Gaussian periods]{The graphic nature of Gaussian periods}

    \author[W.~Duke]{William Duke}
    \address{   Department of Mathematics\\
            UCLA\\
            Los Angeles, California\\
            90095-1555 \\ USA}
    \email{wduke@ucla.edu}
    \urladdr{\url{http://www.math.ucla.edu/~wdduke/}}

    \author[S.~R.~Garcia]{Stephan Ramon Garcia}
    \address{   Department of Mathematics\\
            Pomona College\\
            Claremont, California\\
            91711 \\ USA}
    \email{Stephan.Garcia@pomona.edu}
    \urladdr{\url{http://pages.pomona.edu/~sg064747}}
    \thanks{Partially supported by National Science
Foundation Grant DMS-1001614.}

    \author{Bob Lutz}
    \curraddr{\textsc{Department of Mathematics\\University of Michigan\\
2074 East Hall\\
530 Church Street\\
Ann Arbor, MI  48109-1043}}
    \email{boblutz@umich.edu}

    \begin{abstract}
Recent work has shown that the study of supercharacters on Abelian groups provides 
a natural framework within which to study certain exponential sums of interest in number theory.
Our aim here is to initiate the study of Gaussian periods from this novel perspective.
Among other things, our approach reveals that these classical objects display dazzling visual
patterns of great complexity and remarkable subtlety.    	
    \end{abstract}

\maketitle

\section{Introduction}

	The theory of supercharacters, which generalizes  classical character theory, was recently
	introduced in an axiomatic fashion by P.~Diaconis and I.M.~Isaacs \cite{Diaconis-Isaacs}, extending the seminal work of
	C.~Andr\'e \cite{An95, An01, An02}.   Recent work has shown that the study of supercharacters on Abelian groups provides 
	a natural framework within which to study the properties of certain exponential sums of interest in number theory \cite{RSS, SESUP}
	(see also \cite{CKS}).
	In particular, Gaussian periods, Ramanujan sums, Kloosterman sums, and Heilbronn sums 
	can be realized in this way (see Table \ref{TableNT}).
	Our aim here is to initiate the study of Gaussian periods from this novel perspective.
	Among other things, this approach reveals that these classical objects display a dazzling array of visual
	patterns of great complexity and remarkable subtlety (see Figure \ref{fig:nicepics1}).

	Let $G$ be a finite group with identity 0, $\mathcal{K}$ a partition of $G$, and $\mathcal{X}$ a partition 
	of the set $\Irr(G)$ of irreducible characters of $G$. The ordered pair $(\mathcal{X},\mathcal{K})$ is called a 
	\emph{supercharacter theory} for $G$ if $\{0\}\in\mathcal{K}$, $|\mathcal{X}|=|\mathcal{K}|$, and 
	for each $X\in\mathcal{X}$, the generalized character
		\begin{equation*}
			\sigma_X=\sum_{\chi\in X} \chi(0)\chi
		\end{equation*}
	is constant on each $K\in\mathcal{K}$.
	The characters $\sigma_X$ are called \emph{supercharacters} of $G$ and
	the elements of $\mathcal{K}$ are called \emph{superclasses}.

	Let $G=\Z/n\Z$ and recall that the irreducible characters of $\Z/n\Z$ 
	are the functions $\chi_x(y)=e(\frac{xy}{n})$ for $x$ in $\Z/n\Z$, where $e(\theta)=\operatorname{exp}(2\pi i \theta)$. 
	For a fixed subgroup $A$ of $(\Z/n\Z)^{\times}$, let $\mathcal{K}$ denote the partition of $\Z/n\Z$ arising from the action
	$a\cdot x = ax$ of $A$. The action $a \cdot \chi_x = \chi_{a^{-1}x}$ of $A$ on the irreducible characters of $\Z/n\Z$
	yields a compatible partition $\mathcal{X}$. The reader can verify that $(\mathcal{X},\mathcal{K})$
	is a supercharacter theory on $\Z/n\Z$ and that the corresponding supercharacters are
	given by 
	\begin{equation}\label{eq-sigma}
		\sigma_X(y) = \sum_{x \in X} e\left( \frac{xy}{n} \right),
	\end{equation}
	where $X$ is an orbit in $\Z/n\Z$
	under the action of a subgroup $A$ of $(\Z/n\Z)^{\times}$. 
	When $n=p$ is an odd prime, \eqref{eq-sigma} is a \emph{Gaussian period}, a central object in the theory of cyclotomy. 
	For $p = kd+1$, Gauss defined the \emph{$d$-nomial periods} $\eta_j = \sum_{\ell=0}^{d-1} \zeta_p^{g^{k\ell+j}}$,
	where $\zeta_p = e(\frac{1}{p})$ and $g$ denotes a primitive root modulo $p$ \cite{berndt, davenport, thaine}. 
	Clearly $\eta_j $ runs over the same values as $\sigma_X(y)$ when $y \neq 0$, $|A| = d$, and
	$X = A1$ is the $A$-orbit of $1$.
	For composite moduli, the functions 
	$\sigma_X$ attain values which are generalizations of Gaussian periods of the type considered by Kummer
	and others (see \cite{Lehmers}). 

	\begin{figure}[H]
		\centering
		\begin{subfigure}[b]{0.3\textwidth}
	                \centering
	                \includegraphics[width=\textwidth]{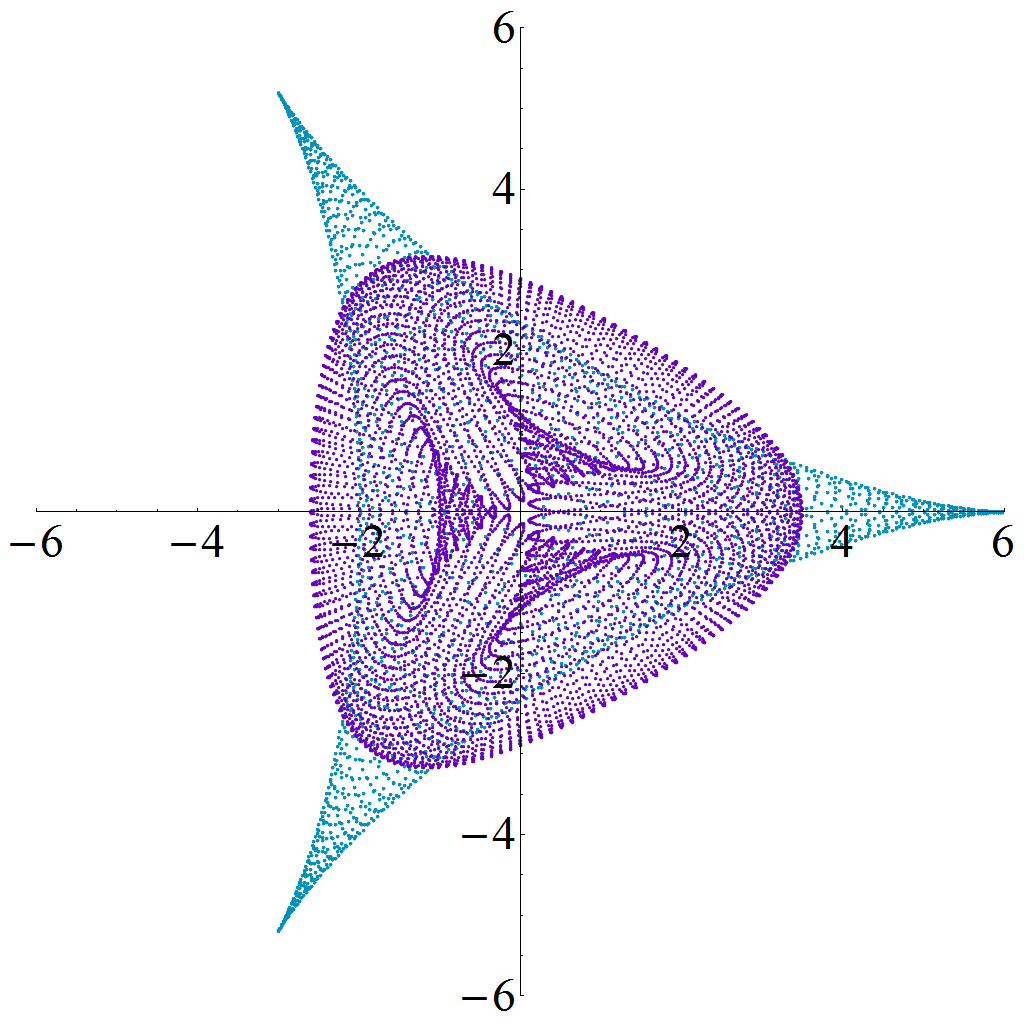}
	                \caption{{\scriptsize $n=52059$, $A=\langle766\rangle$}}
	                \label{fig:52059_766}
	        \end{subfigure}
	        \quad
	        		\begin{subfigure}[b]{0.3\textwidth}
	                \centering
	                \includegraphics[width=\textwidth]{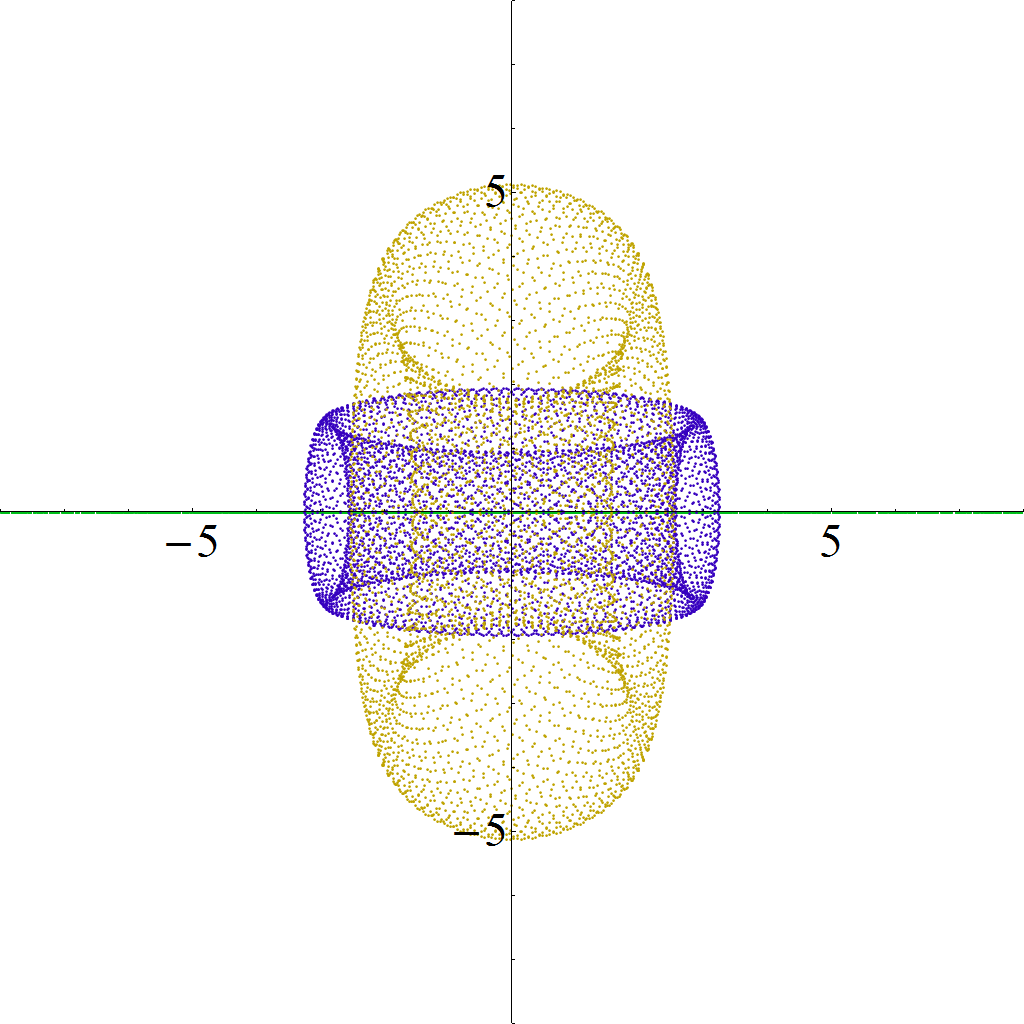}
	                \caption{{\scriptsize $n=91205$, $A=\langle2337\rangle$}}
	                \label{fig:20485_4643}
	        \end{subfigure}
	        \quad
	        \begin{subfigure}[b]{0.3\textwidth}
	                \centering
	                \includegraphics[width=\textwidth]{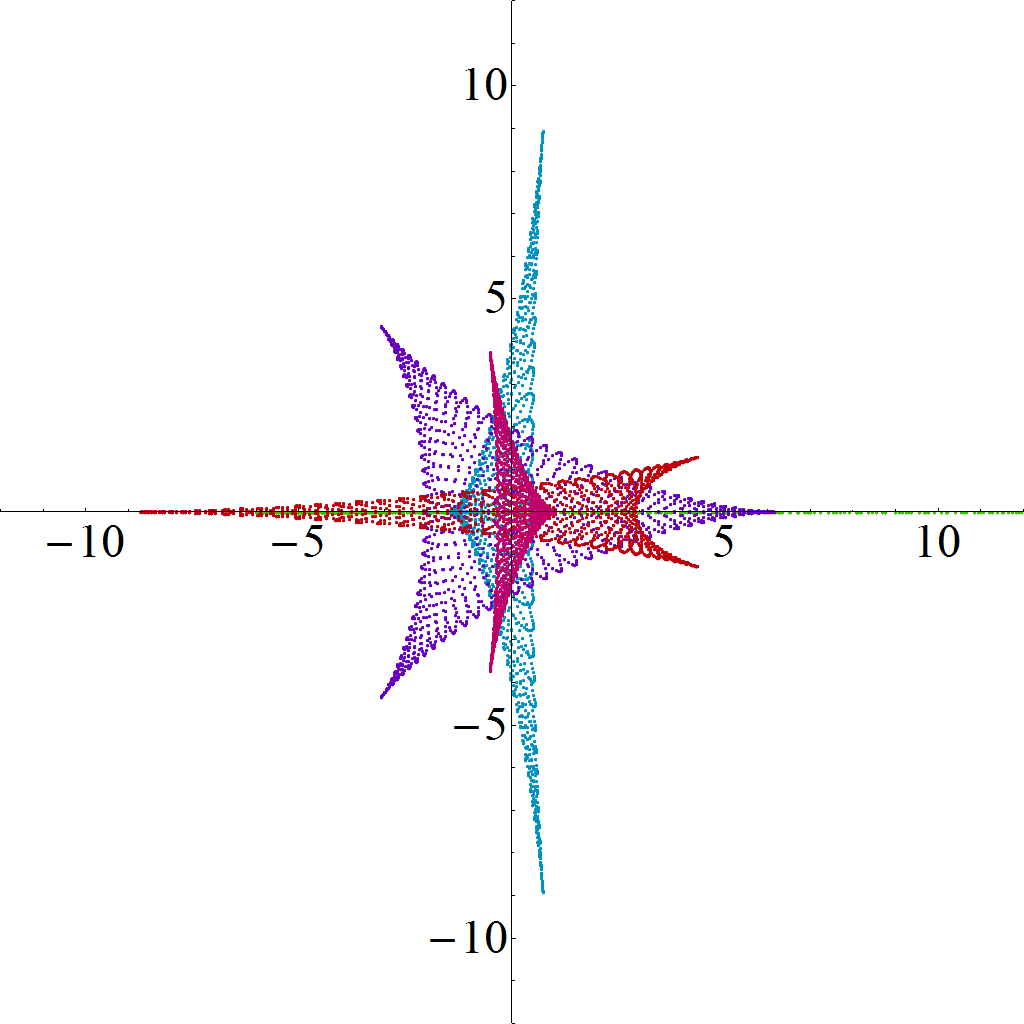}
	                \caption{{\scriptsize $n=70091$, $A=\langle 3447\rangle$}}
	                \label{fig:39711_4283}
	        \end{subfigure}
			\\
	        \begin{subfigure}[b]{0.3\textwidth}
	                \centering
	                \includegraphics[width=\textwidth]{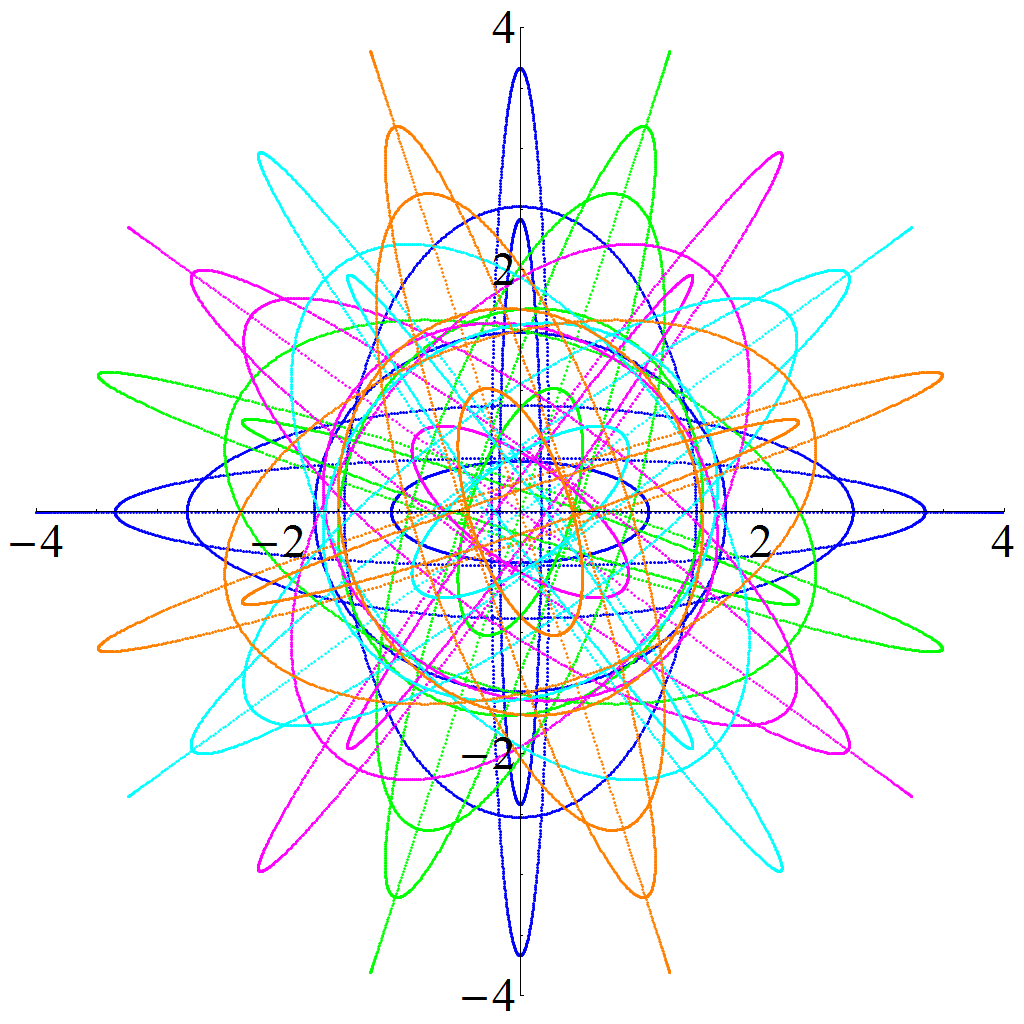}
	                \caption{{\scriptsize $n=91205$, $A=\langle 39626\rangle$}}
	                \label{fig:30030_6313}
	        \end{subfigure}
			\quad
	        \begin{subfigure}[b]{0.3\textwidth}
	                \centering
	                \includegraphics[width=\textwidth]{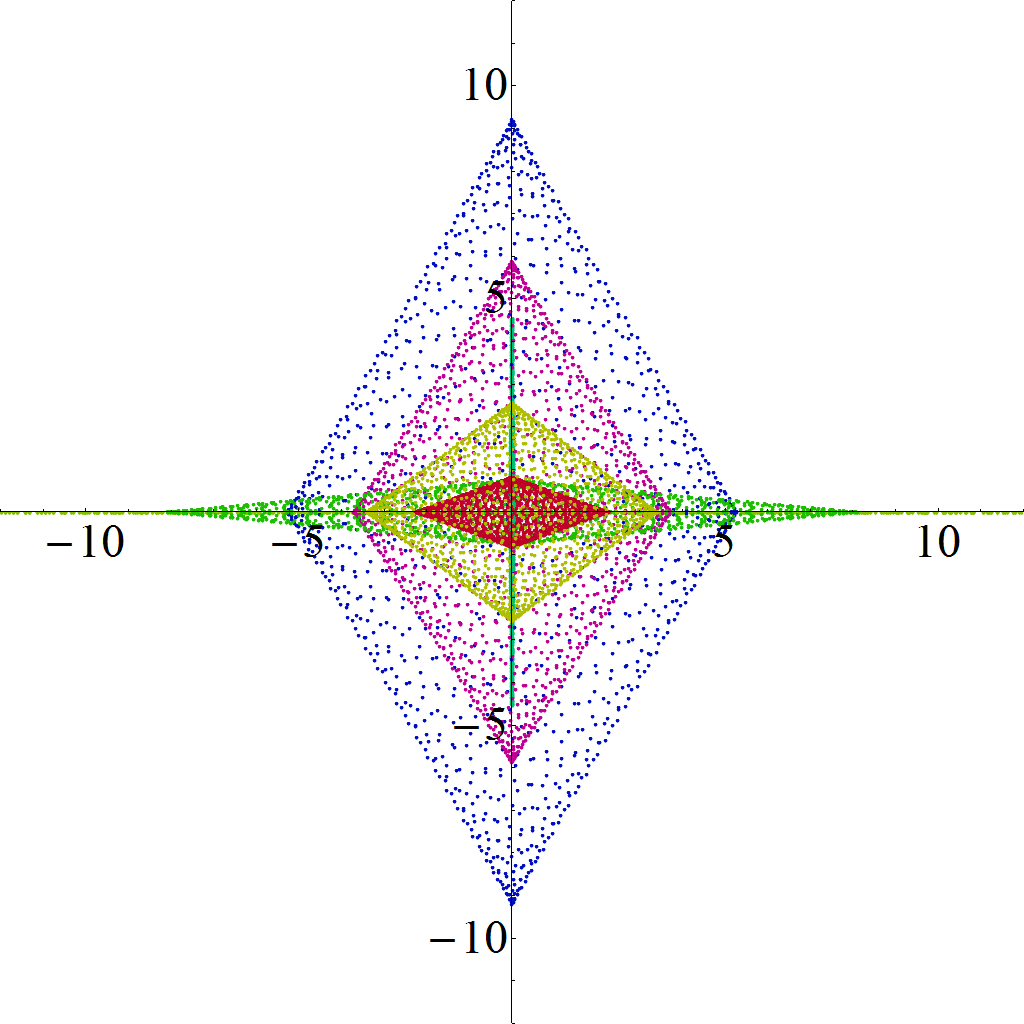}
	                \caption{{\scriptsize $n=91205$, $A=\langle 1322\rangle$}}
	                \label{fig:91205_1322}
	        \end{subfigure}
			\quad
			\begin{subfigure}[b]{0.3\textwidth}
	                \centering
	                \includegraphics[width=\textwidth]{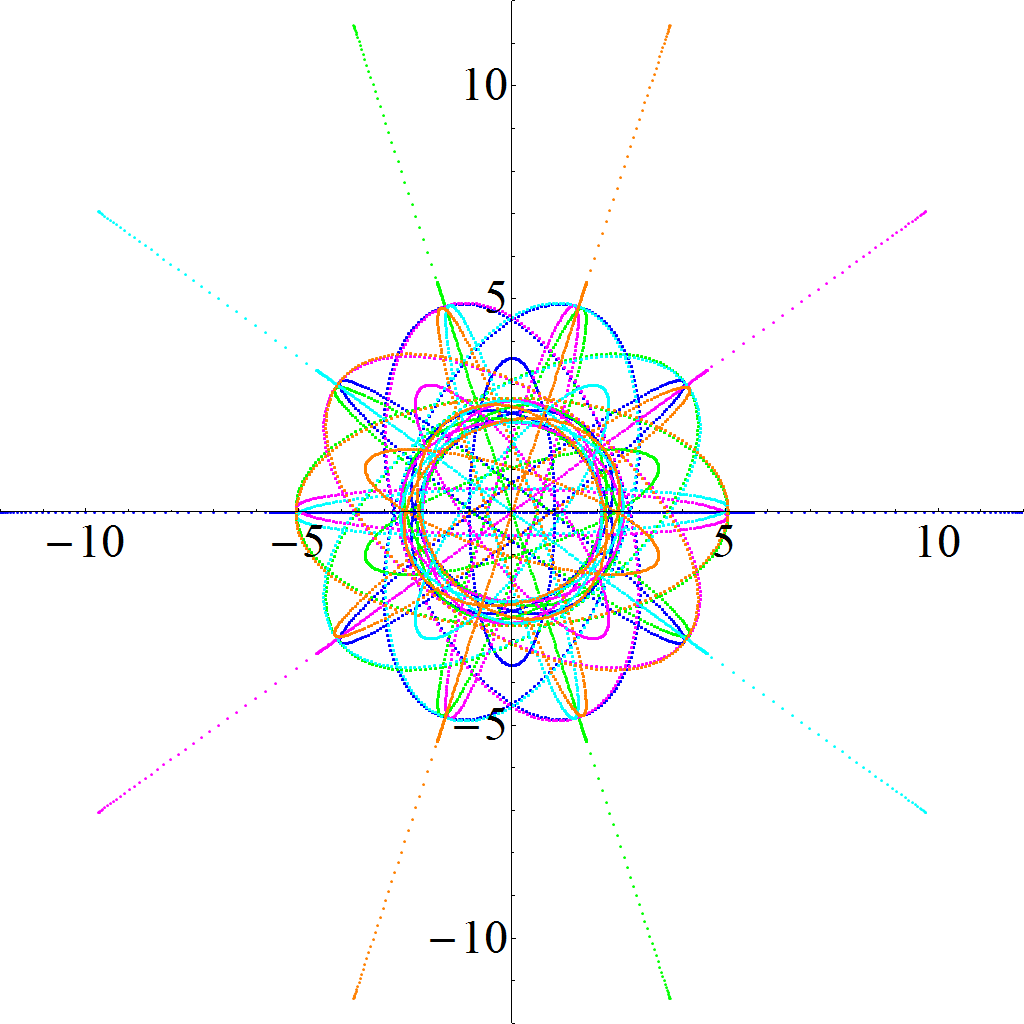}
	                \caption{{\scriptsize $n=95095$, $A=\langle 626\rangle$}}
	                \label{fig:21505_1011}
	        \end{subfigure}
	        \\
	                \begin{subfigure}[b]{0.3\textwidth}
                \centering
                \includegraphics[width=\textwidth]{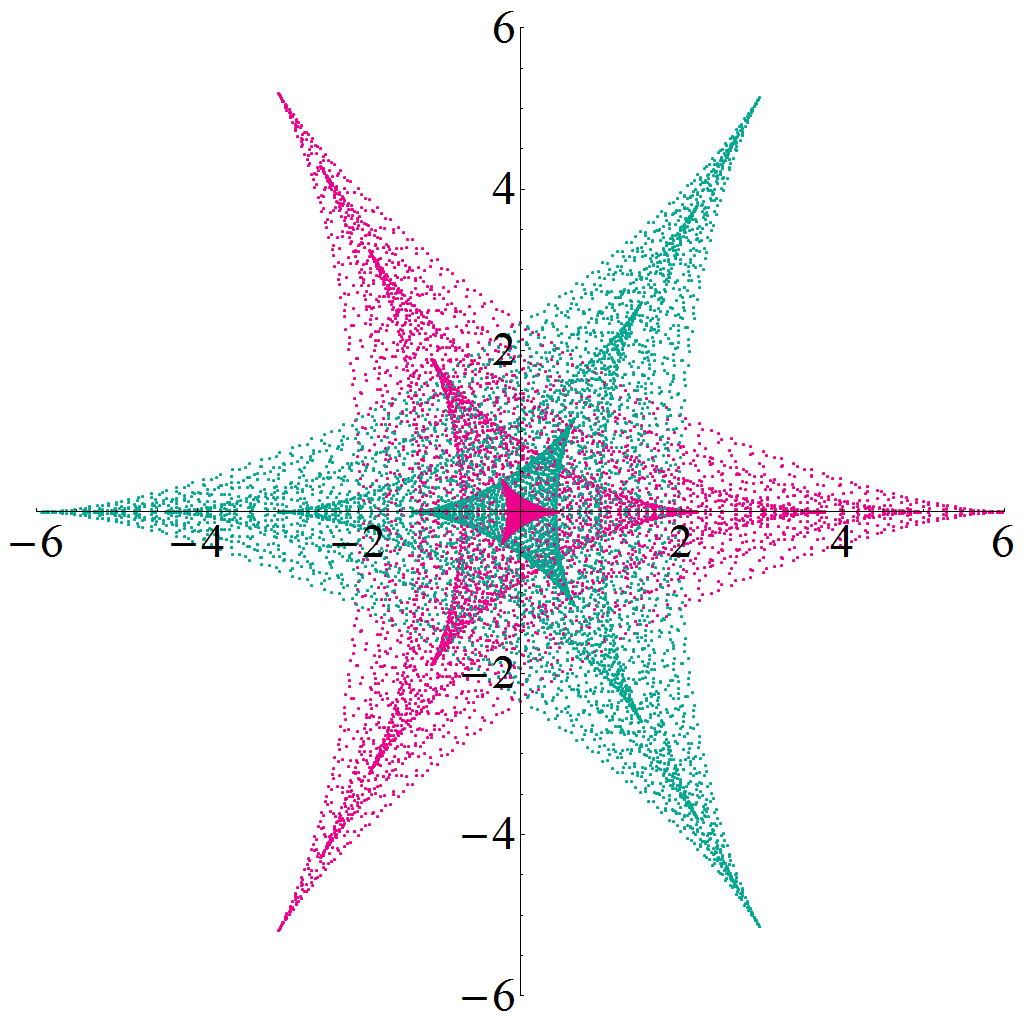}
                \caption{{\scriptsize $n=82677$, $A=\langle 8147\rangle$}}
                \label{fig:39711_9956}
        \end{subfigure}
		\quad
		\begin{subfigure}[b]{0.3\textwidth}
                \centering
                \includegraphics[width=\textwidth]{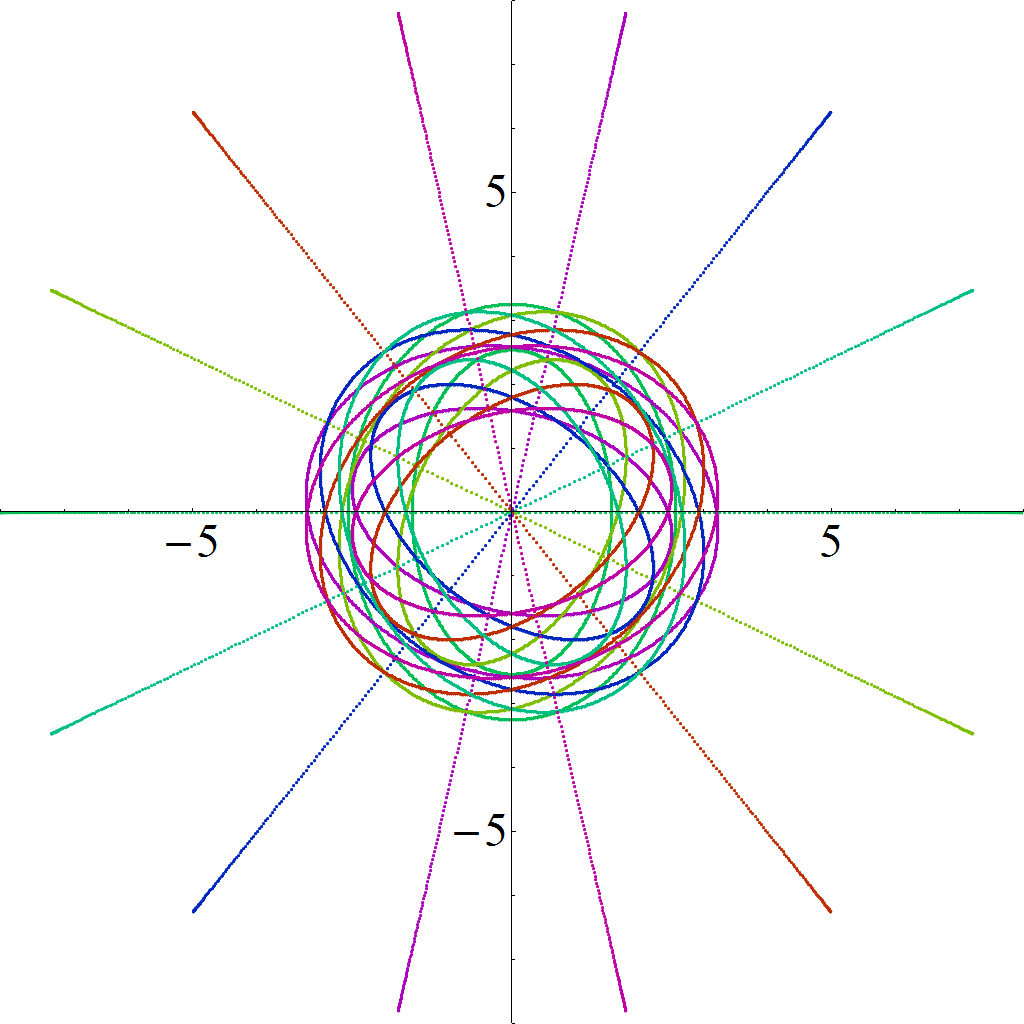}
                \caption{{\scriptsize $n=70091$, $A=\langle21792\rangle$}}
                \label{fig:20485_4113}
        \end{subfigure}
		\quad
		\begin{subfigure}[b]{0.3\textwidth}
                \centering
                \includegraphics[width=\textwidth]{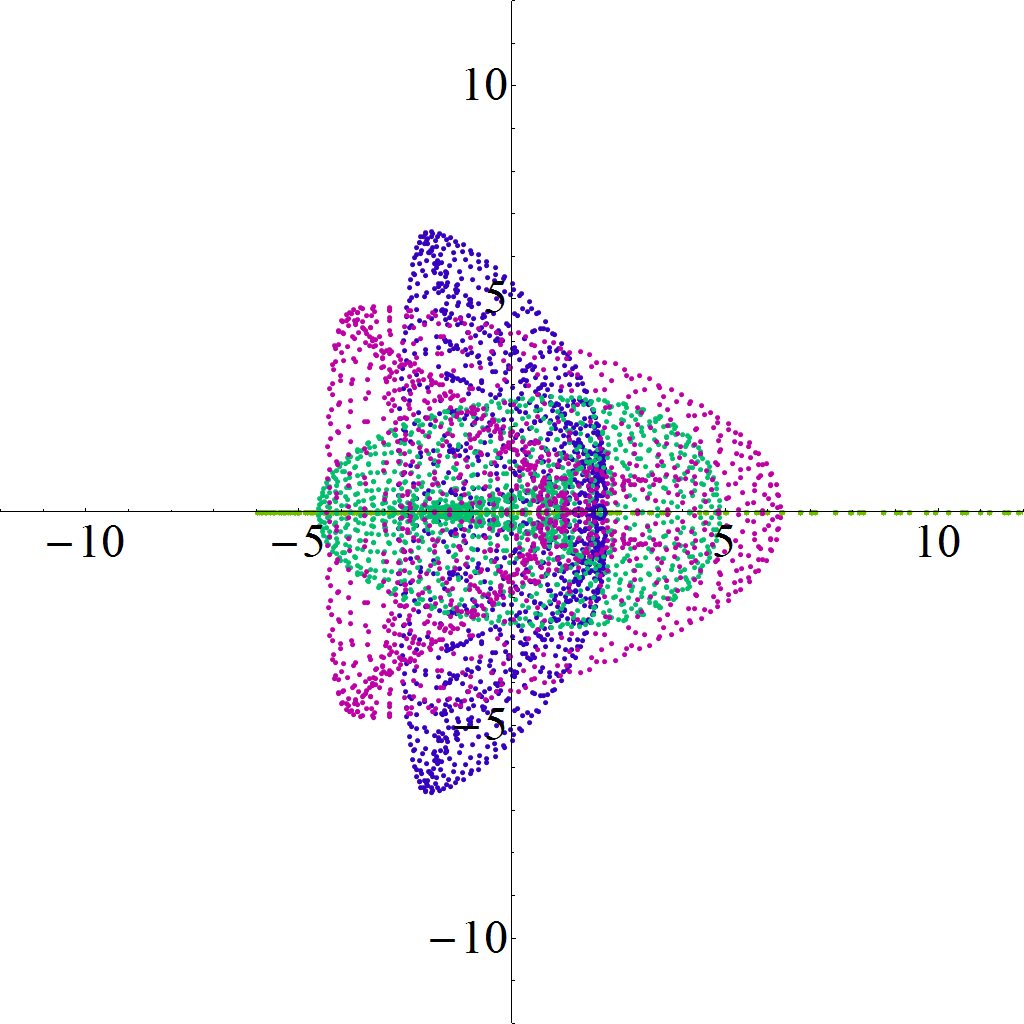}
                \caption{{\scriptsize $n=51319$, $A=\langle 430\rangle$}}
                \label{fig:21385_281}
        \end{subfigure}
	        \caption{Each subfigure is the image of $\sigma_X:\Z/n\Z\to\C$, where $X$ is the orbit of 
	        $1$ under the action of a unit subgroup $A$ on $\Z/n\Z$. If $\sigma_X(y)$ and $\sigma_X(y')$ differ in color, then $y\nequiv y'\pmod{m}$, where $m$ is a certain fixed proper divisor of $n$. Coloring the points $\sigma_X(y)$ according to residue classes of $y$ can reveal hidden structure.}
		\label{fig:nicepics1}
	\end{figure}

	\begin{table}
		\begin{equation*}\footnotesize
		\begin{array}{|c|c|c|c|}
			\hline
			\text{Name} & \text{Expression} & G & A \\
			\hline\hline
			\text{Gauss} & 
			\eta_j = \displaystyle \sum_{\ell=0}^{d-1} e\left( \frac{g^{k\ell+j}}{p}\right)
			& \Z/p\Z & \text{nonzero $k$th powers mod $p$} \\[20pt]
			\text{Ramanujan} & c_n(x)=\displaystyle \sum_{ \substack{ j = 1 \\ (j,n) = 1} }^n \!\!\!\! 
			e\left(\frac{ jx}{n} \right) & \Z/n\Z & (\Z/n\Z)^{\times} \\[20pt]
			\text{Kloosterman} & K_p(a,b)=\displaystyle \sum_{ \ell = 0  }^{p-1} e\left( \frac{a\ell + b \overline{\ell} }{p}\right)
			 & (\Z/p\Z)^2 & \left\{ \minimatrix{u}{0}{0}{u^{-1}}  : u \in (\Z/p\Z)^{\times} \right\} \\[20pt]
			\text{Heilbronn} & \displaystyle H_p(a)=\sum_{\ell=0}^{p-1} e\left(\frac{a \ell^p}{p^2} \right) & \Z/p^2\Z & 
			\footnotesize\text{nonzero $p$th powers mod $p^2$} \\[20pt]
			\hline
		\end{array}
		\end{equation*}	
		\caption{Gaussian periods, Ramanujan sums, Kloosterman sums, and Heilbronn sums 
		appear as supercharacters arising from the action of a subgroup $A$ of $\Aut G$ for a suitable
		abelian group $G$.  Here $p$ denotes an odd prime number.}
		\label{TableNT}
	\end{table}

	When visualized as subsets of the complex plane, the images of these supercharacters
	exhibit a surprisingly diverse range of features (see Figure \ref{fig:nicepics1}).
	The main purpose of this paper is to initiate the investigation of these plots, focusing our
	attention on the case where $A = \langle a \rangle$ is a cyclic subgroup of $(\Z/n\Z)^{\times}$.
	We refer to supercharacers which arise in this manner as \emph{cyclic supercharacters}.

	The sheer diversity of patterns displayed by cyclic supercharacters is overwhelming.  
	To some degree, these circumstances force us to focus our initial efforts on documenting the 
	notable features that appear and on explaining their number-theoretic origins.  
	One such theorem is the following.
		
	\begin{thm}\label{TheoremHypo}
		Suppose that $q$ is an odd prime power and that $\sigma_X$ is a cyclic supercharacter of $\Z/q\Z$. 
		If $|X|=d$ is prime, then the image of $\sigma_X$ is bounded by the $m$-cusped hypocycloid 
		parametrized by $\theta\mapsto(d-1)e^{i\theta} + e^{i(d-1)\theta}$.
	\end{thm}

	In fact, for a fixed prime $m$, as the modulus $q\equiv 1 \pmod{d}$ tends to infinity the corresponding
	supercharacter images become dense in the filled hypocycloid in a sense that will be made precise
	in Section \ref{SectionAsymptotic}.

	\begin{figure}[h]
		\centering
	        \begin{subfigure}[b]{0.30\textwidth}
	                \centering
	                \includegraphics[width=\textwidth]{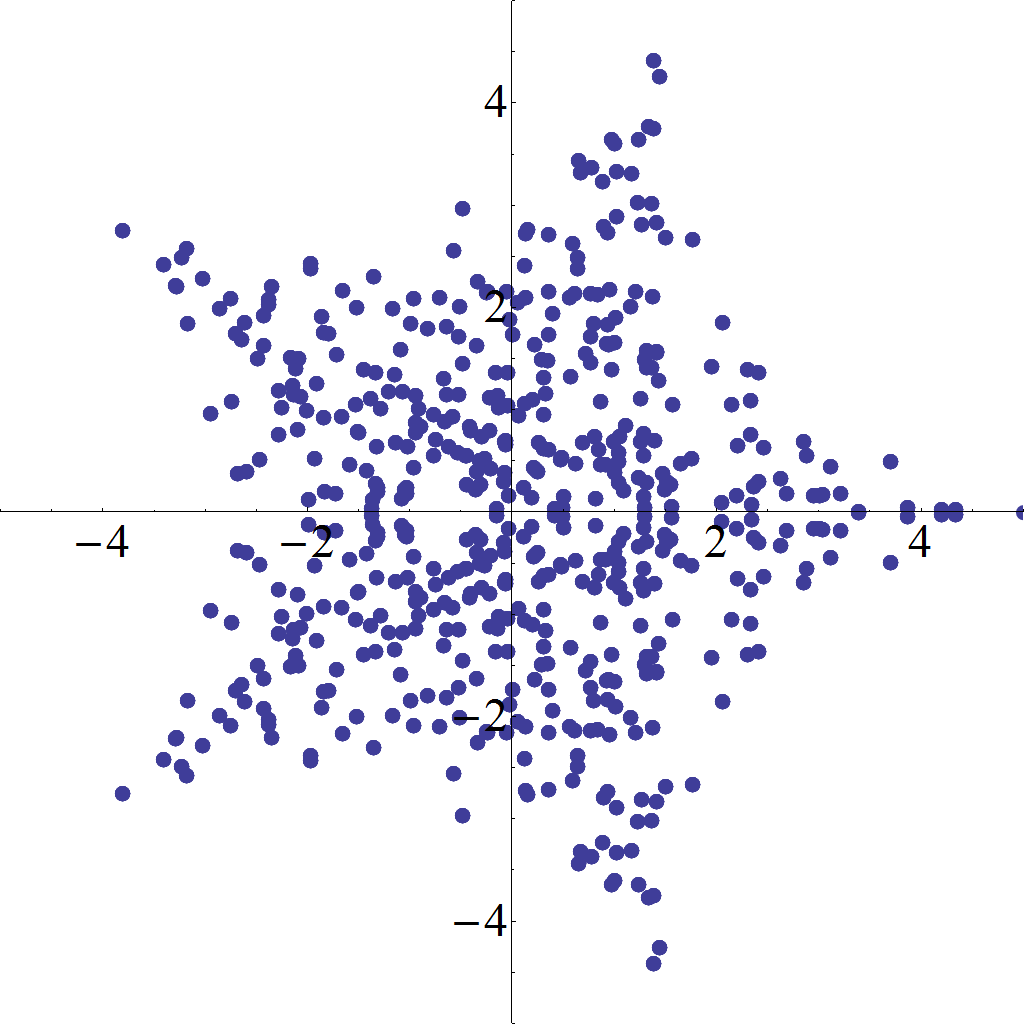}
	                \caption{{\scriptsize $p=2791$, $A=\langle 800\rangle$}}
	        \end{subfigure}
	        \quad
	        \begin{subfigure}[b]{0.30\textwidth}
	                \centering
	                \includegraphics[width=\textwidth]{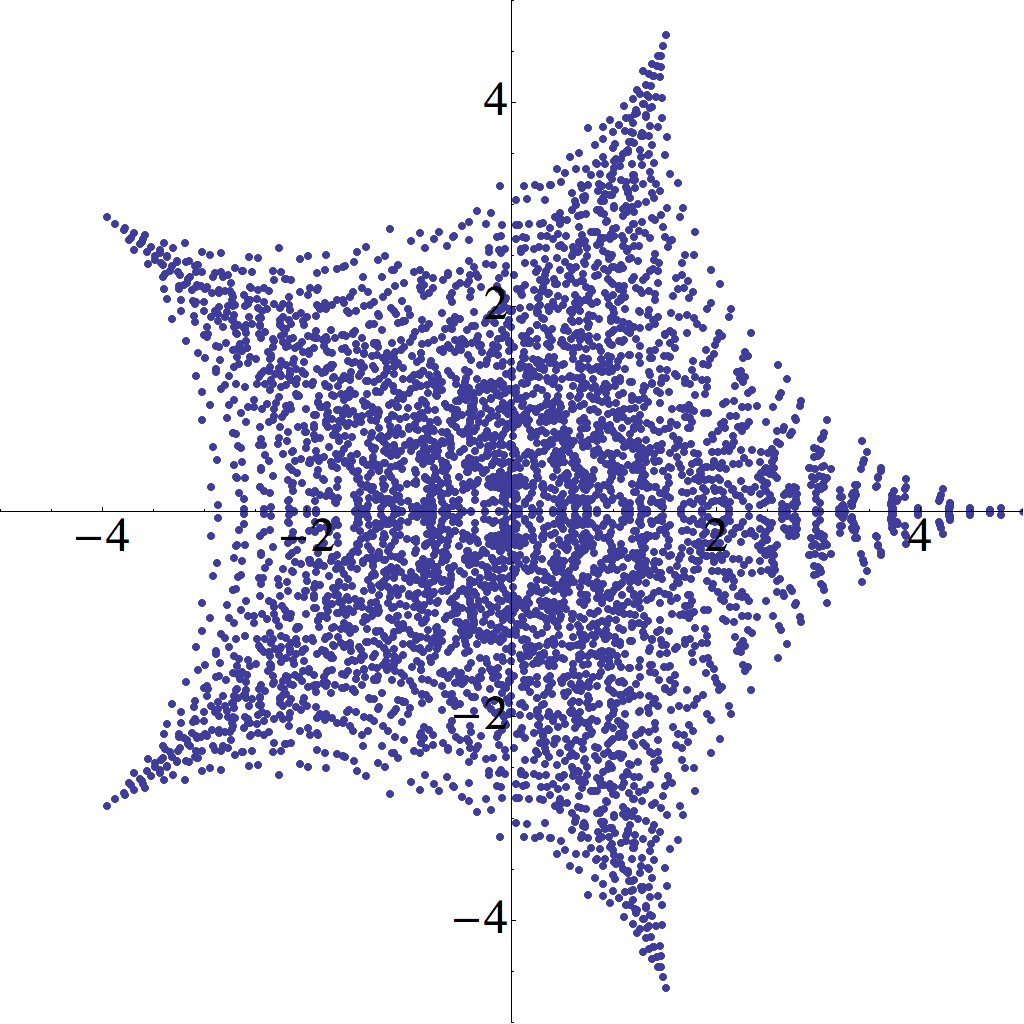}
	                \caption{{\scriptsize $p=27011$, $A=\langle 9360\rangle$}}
	        \end{subfigure}
		\quad
	        \begin{subfigure}[b]{0.30\textwidth}
	                \centering
	                \includegraphics[width=\textwidth]{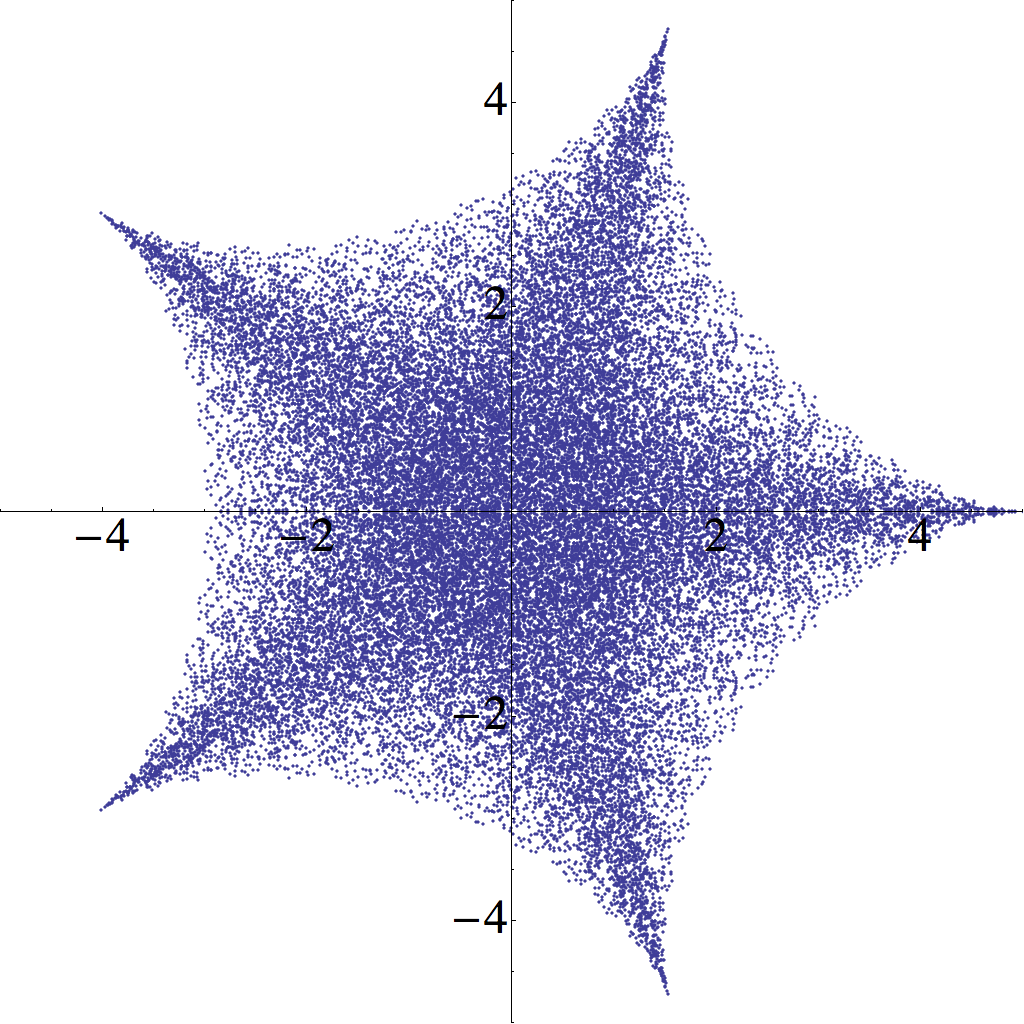}
	                \caption{{\scriptsize $p=202231$, $A=\langle 61576\rangle$}}
	        \end{subfigure}
	        \caption{Graphs of cyclic supercharacters $\sigma_X$ of $\Z/q\Z$, where $X=A1$, showing the density of hypocycloids as $q\to\infty$.}
	\end{figure}

	The preceding theorem is itself a special case of a much more general theorem
	(Theorem \ref{TheoremMain}) which relates the asymptotic behavior of 
	cyclic supercharacter plots to the mapping properties of certain multivariate Laurent polynomials,
	regarded as complex-valued functions on a suitable, high-dimensional torus.

\section{Multiplicativity and nesting plots}\label{SectionNesting}

	Our first order of business is to determine when and in what manner the image of one cyclic supercharacter plot can appear in another. Certain cyclic supercharacters have a naturally multiplicative structure. When combined with Proposition \ref{primescale} and the discussion in Section \ref{SectionAsymptotic}, the following result provides a complete picture of the boundaries of these supercharacters. Following the introduction, we let $X = Ar$ denote the orbit of $r$ in $\Z/n\Z$ under the action of a cyclic unit subgroup $A$.

\begin{thm}\label{thm-mult}
Let $\sigma_X$ be a cyclic supercharacter of $\Z/n\Z$, writing $n=\prod_{j=1}^k p_j^{a_j}$ in standard form and $X=\langle\omega\rangle r$. For each $j$, let $\psi_j:\Z/n\Z\to\Z/p_j^{a_j}\Z$ be the natural homomorphism, let $x_j$ be the multiplicative inverse of $n/p_j^{a_j}\pmod{p_j^{a_j}}$, and write $X_j = \langle \psi_j(\omega)\rangle x_j\psi_j(r)$. If the orbit sizes $|X_j|$ are pairwise coprime, then
$$\sigma_X(y) = \prod_{j=1}^k \sigma_{X_j}(\psi_j(y)).$$
\begin{proof}
We prove the theorem for $n=p_1p_2$ a product of distinct primes; the general argument is similar. Let $\psi=(\psi_1,\psi_2)$ be the ring isomorphism given by the Chinese Remainder Theorem, and let $d=|X|$, $d_1=|\psi_1(X)|$ and $d_2=|\psi_2(X)|$. We have
\begin{align*}
\sigma_{X_1}(\psi_1(y))\sigma_{X_2}(\psi_2(y))
&=\sum_{j=0}^{d_1-1} e\left(\frac{\psi_1(\omega^jry) x_1}{p_1}\right)\sum_{k=0}^{d_2-1} e\left(\frac{\psi_2(\omega^k ry) x_2}{p_2}\right)\\
&=\sum_{j=0}^{d_1-1}\sum_{k=0}^{d_2-1} e\left(\frac{\psi_1(\omega^jry) x_1}{p_1}+\frac{\psi_2(\omega^k ry) x_2}{p_2}\right)\\
&=\sum_{j=0}^{d_1-1}\sum_{k=0}^{d_2-1}e\left(\frac{\psi^{-1}(\phi_1(\omega^j ry),\psi_2(\omega^k ry))}{n}\right)\\
&=\sum_{j=0}^{d_1-1}\sum_{k=0}^{d_2-1}e\left(\frac{\psi^{-1}(\phi_1(\omega)^j,\psi_2(\omega)^k)ry}{n}\right)\\
&=\sum_{\ell=0}^{d-1}e\left(\frac{\omega^\ell ry}{n}\right)\\
&=\sigma_X(y).\qedhere
\end{align*}
\end{proof}
\end{thm}

	As a consequence of the next result, we observe all possible graphical behavior, up to scaling, 
	by restricting our attention to cases where $r=1$ (i.e., where $X=A$ as sets). We present it without proof.

	\begin{prop}\label{thm:master}
		Let $r$ belong to $\Z/n\Z$, and suppose that $(r,n)=\tfrac{n}{d}$ for some positive divisor $d$ of $n$, 
		so that $\xi=\frac{rd}{n}$ is a unit modulo $n$. Also let $\psi : \Z/n\Z\to\Z/d\Z$ be the natural homomorphism.
		\begin{enumerate}
			\item[(i)] The images of $\sigma_{Ar}$, $\sigma_{A(r,n)}$, and $\sigma_{\psi_d(A)1}$ are equal.
			\item[(ii)] The image in (i), when scaled by $\frac{|A|}{|\psi_d(A)|}$, is a subset of the image of $\sigma_{A\xi}$.
		\end{enumerate}
	\end{prop}

	\begin{figure}[h]
		\centering
		\begin{subfigure}[b]{0.3\textwidth}
	                \centering
	                \includegraphics[width=\textwidth]{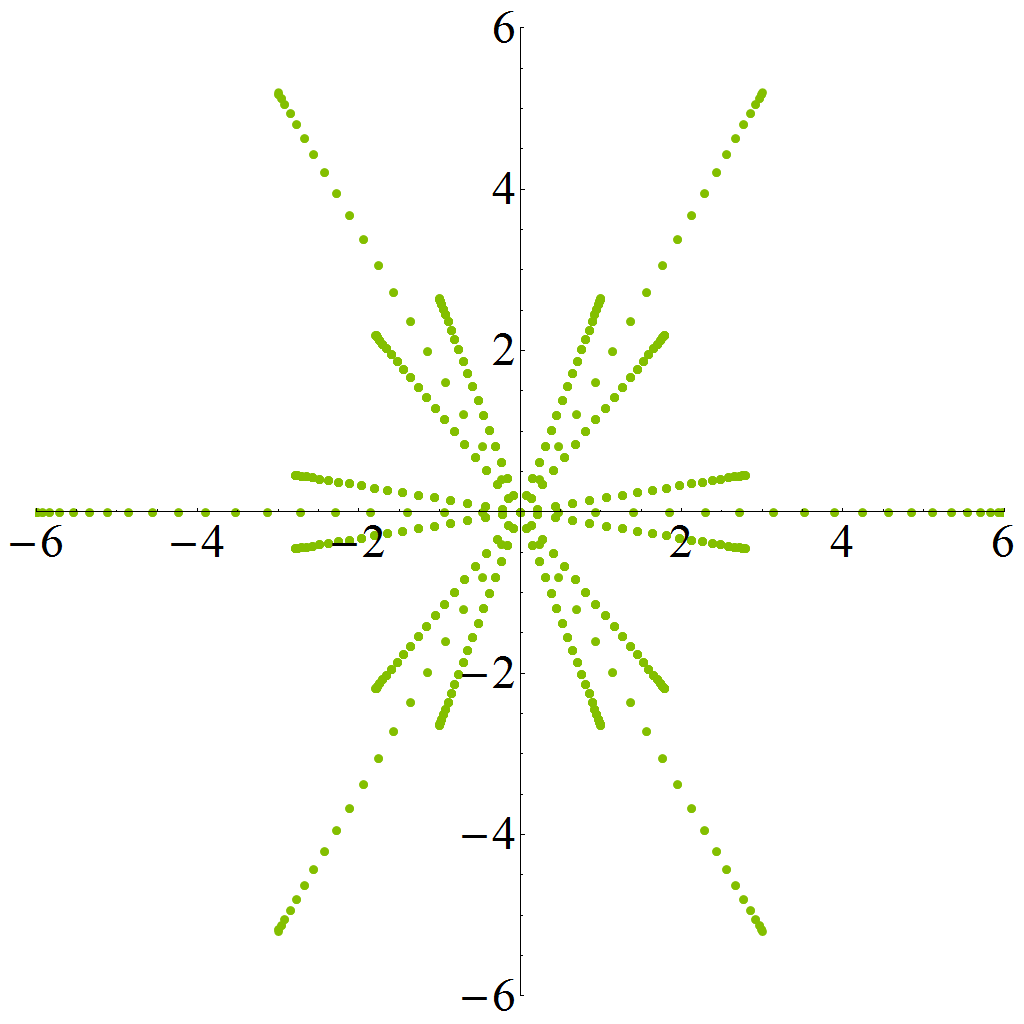}
	                \caption{{\scriptsize $r=37$}}\label{fig:varyingd_2}
	        \end{subfigure}
			\quad
	        \begin{subfigure}[b]{0.3\textwidth}
	                \centering
	                \includegraphics[width=\textwidth]{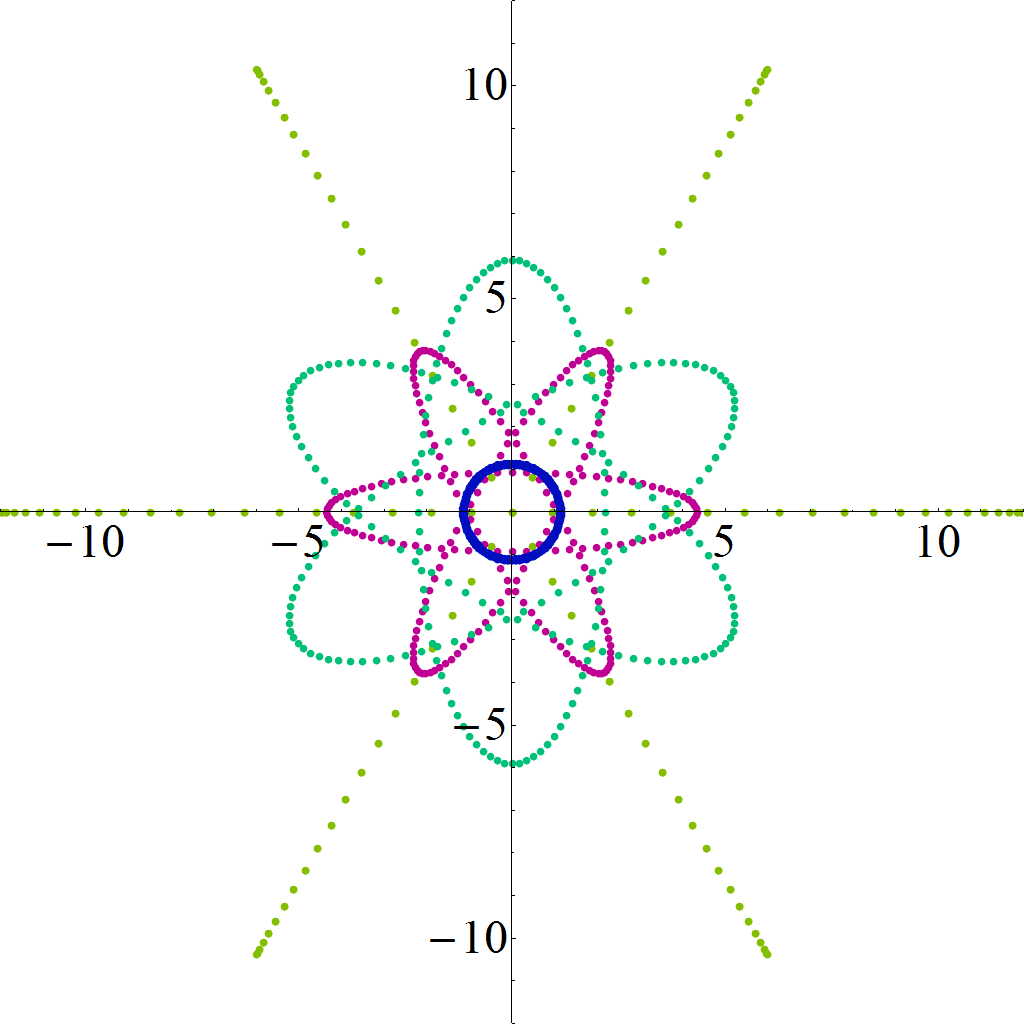}
	                \caption{{\scriptsize $r=7$}}
	                \label{fig:20485_4609}
	        \end{subfigure}
			\quad
	        \begin{subfigure}[b]{0.3\textwidth}
	                \centering
	                \includegraphics[width=\textwidth]{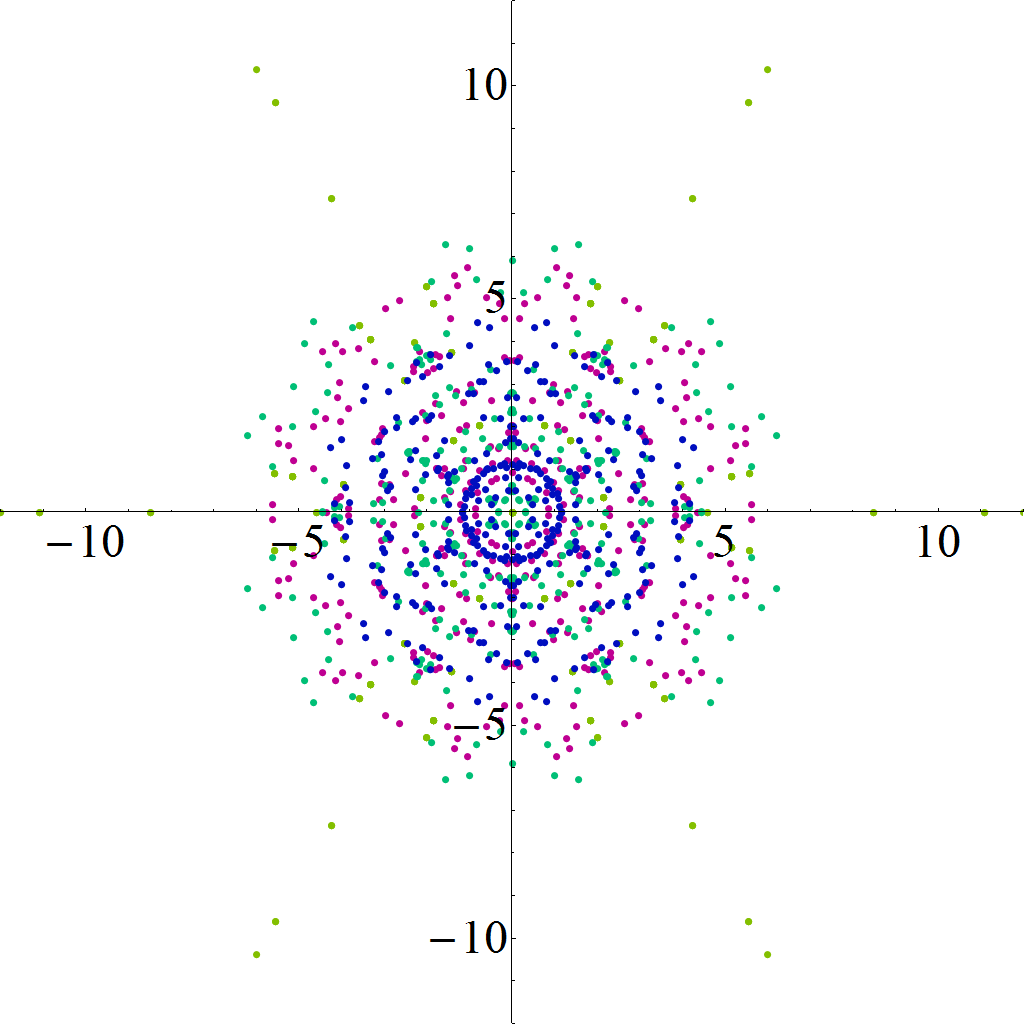}
	                \caption{{\scriptsize $r=5$}}
	                \label{fig:68913_88}
	        \end{subfigure}
			\\
			\begin{subfigure}[b]{0.3\textwidth}
	                \centering
	                \includegraphics[width=\textwidth]{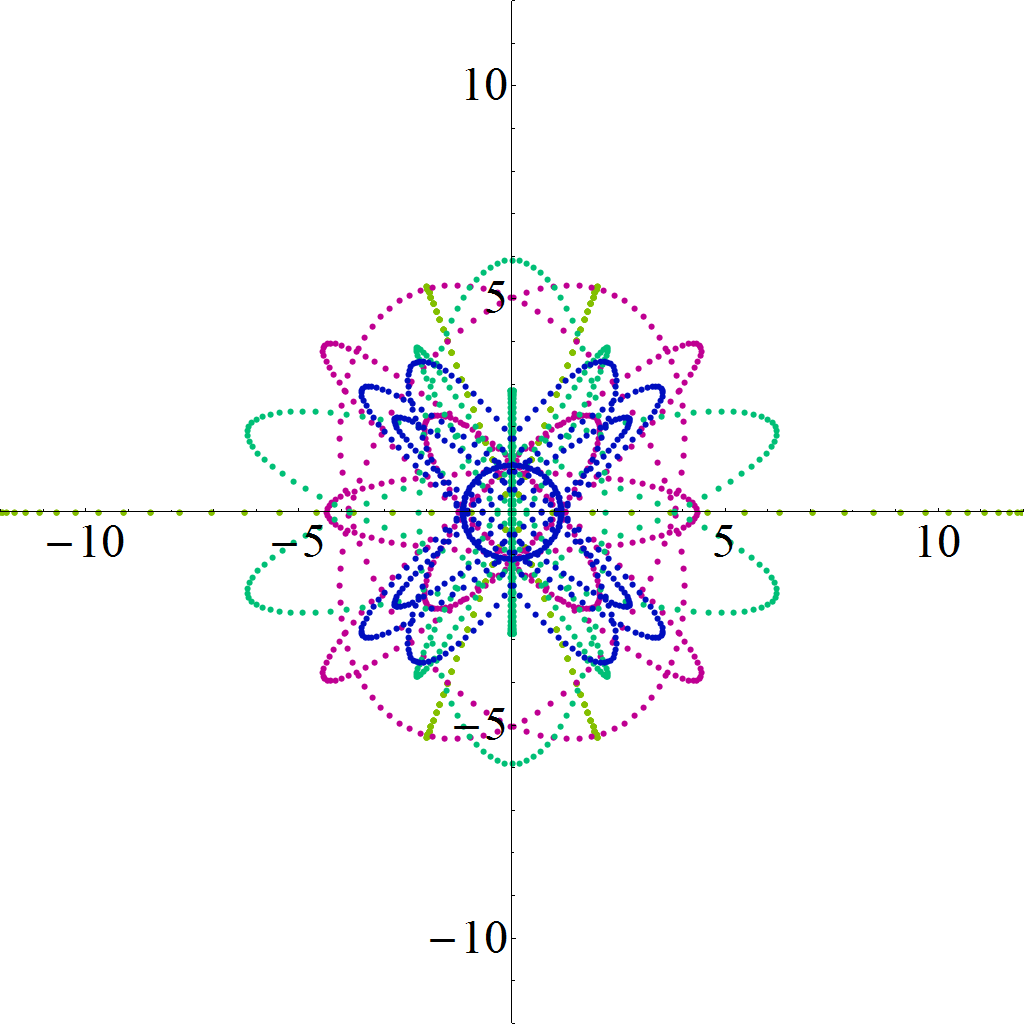}
	                \caption{{\scriptsize $r=3$}}
	                \label{fig:d20720}
	        \end{subfigure}
			\quad
	        \begin{subfigure}[b]{0.3\textwidth}
	                \centering
	                \includegraphics[width=\textwidth]{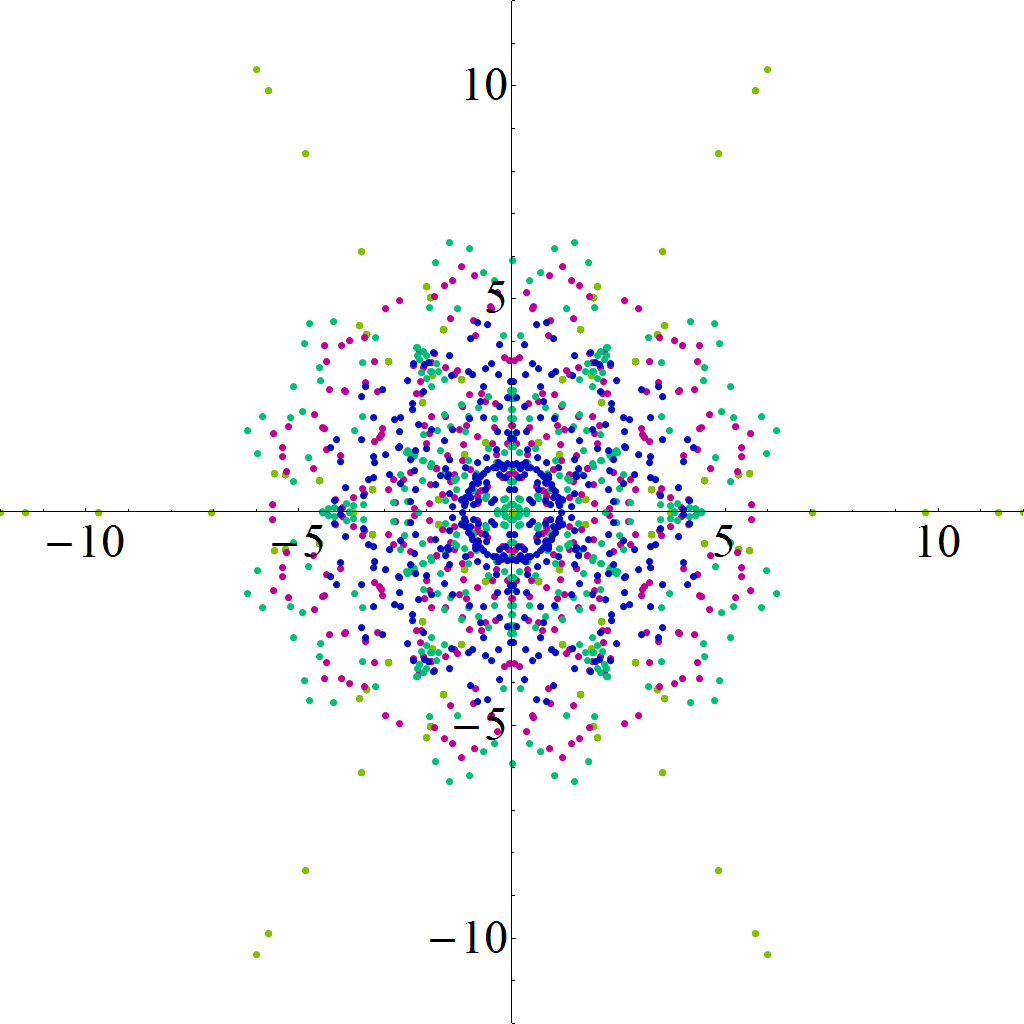}
	                \caption{{\scriptsize $r=4$}}
	                \label{fig:20485_4609}
	        \end{subfigure}
			\quad
			\begin{subfigure}[b]{0.3\textwidth}
	                \centering
	                \includegraphics[width=\textwidth]{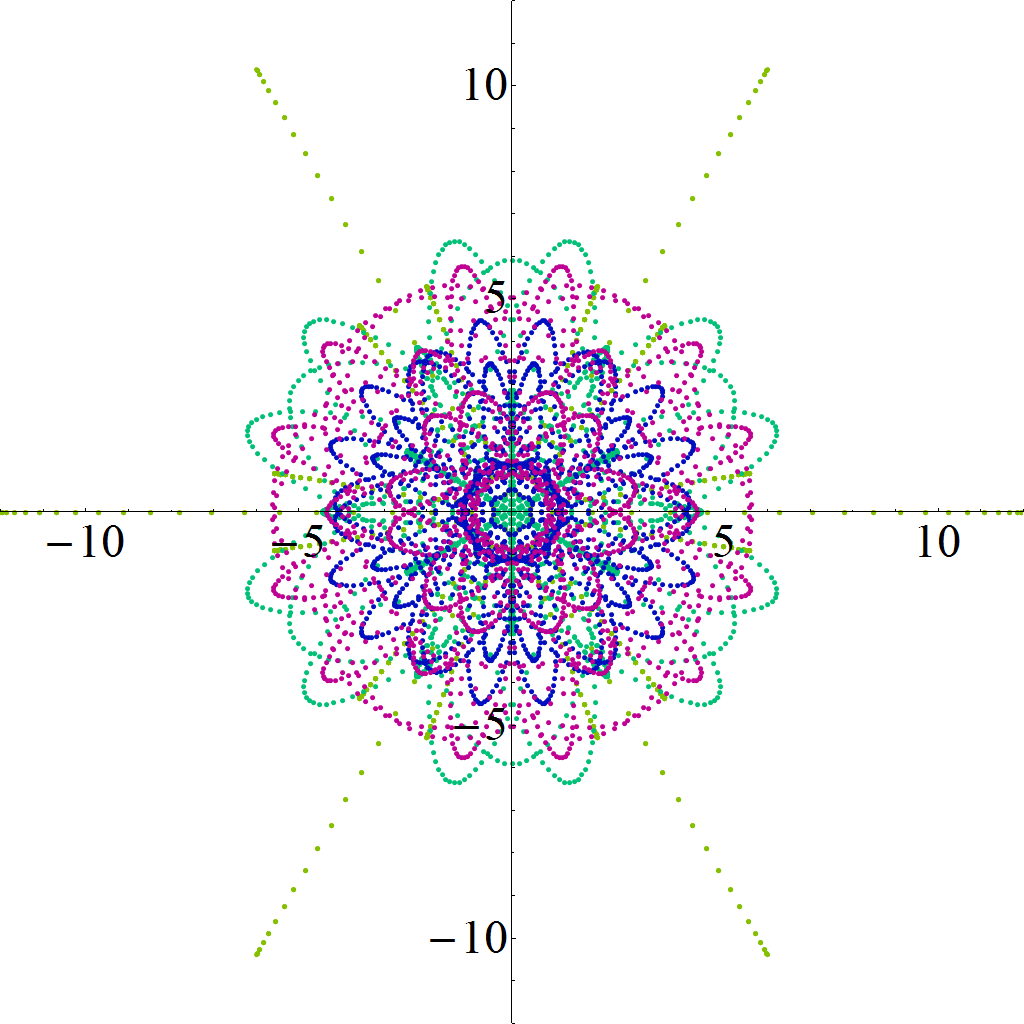}
	                \caption{{\scriptsize $r=1$}}
	                \label{fig:d62160}
	        \end{subfigure}
	        \caption{Graphs of cyclic supercharacters $\sigma_X$ of $\Z/62160\Z$, 
	        where $X=\langle 319\rangle r$. Each image nests in Figure \ref{fig:d62160}, as per Proposition \ref{thm:master}(ii). See Figure \ref{fig:nicepics1} for a brief discussion of colorization.}
	\label{fig:varyingd2}
	\end{figure}

	\begin{example}
		Let $n=62160=2^4\cdot3\cdot5\cdot7\cdot37$. Each plot in Figure \ref{fig:varyingd2} displays 
		the image of a different cyclic supercharacter $\sigma_X$, where $X=\langle319\rangle r$. 
		If $d=r/(n,r)$, then Proposition \ref{thm:master}(i) says that each image equals that of a cyclic supercharacter $\sigma_{X'}$ of $\Z/d\Z$, where 
		$X'=\langle\psi_d(319)\rangle1$. 
		Proposition \ref{thm:master}(ii) says that each nests in the image in Figure \ref{fig:d62160}.
	\end{example}

In part because of Theorem \ref{thm-mult}, we are especially interested in cyclic supercharacters with prime power moduli. The following result implies that the image of any cyclic supercharacter of $\Z/p^a\Z$ is essentially a scaled copy of one whose boundary is given by Theorem \ref{TheoremMain}.

\begin{prop}\label{primescale}
Let $p$ be an odd prime, $a>b$ nonnegative integers, and $\psi$ the natural homomorphism from $\Z/p^a\Z$ to $\Z/p^{a-b}\Z$. If $\sigma_X$ is a cyclic supercharacter of $\Z/p^{a-b}\Z$, where $X=A1$ with $p^b\divides|X|$ and $p^{a-b}\equiv 1\pmod{|\phi(X)|}$, then
$$\sigma_X(\Z/p^a\Z)=\{0\}\cup p^b\sigma_{\phi(X)}(\Z/p^{a-b}\Z).$$

\begin{proof}
Let $k$ be a positive divisor of $p-1$. If $|X|=kp^b$, then $A=\psi^{-1}(A')$, where $A'$ is the unique subgroup of $(\Z/p^{a-b}\Z)^\times$ of order $k$. Let $X'=A'1$ (i.e., $X'=\psi(X)$), so that
$$X=\{x+jp^{a-b} : x\in X',\, j=0,1,\ldots,p^b-1\}.$$
We have
\begin{align*}
\sigma_X(y)&=\sum_{x\in X'} \sum_{j=0}^{p^b-1} e\left(\frac{(x+jp^{a-b})y}{p^a}\right)\\
&=\sum_{j=0}^{p^b-1} e\left(\frac{jy}{p^b}\right) \sum_{x\in X'} e\left(\frac{xy}{p^a}\right)\\
&=\begin{cases} p^b \sigma_{X'}(\psi(y))&\mbox{if } p^b|y,\\0&\mbox{else}.\\\end{cases}\qedhere
\end{align*}
\end{proof}
\end{prop}

\section{Symmetries}\label{SectionSymmetries}
	We say that a cyclic supercharacter $\sigma_X:\Z/n\Z\to\C$ has \emph{$k$-fold dihedral symmetry} 
	if its image is invariant under the natural action of the dihedral group of order $2k$. In other words, 
	$\sigma_X$ has $k$-fold dihedral symmetry if its image is invariant under complex conjugation and rotation by $2\pi/k$ about the origin. If $X$ is the orbit of $r$, where $(r,n)=\frac{n}{d}$ for some odd divisor $d$ of $n$, then $\sigma_X$ is generally asymmetric about the imaginary axis, as evidenced by Figure \ref{fig:axial}.

	\begin{figure}[h]
		\centering
	        \begin{subfigure}[b]{0.30\textwidth}
	                \centering
	                \includegraphics[width=\textwidth]{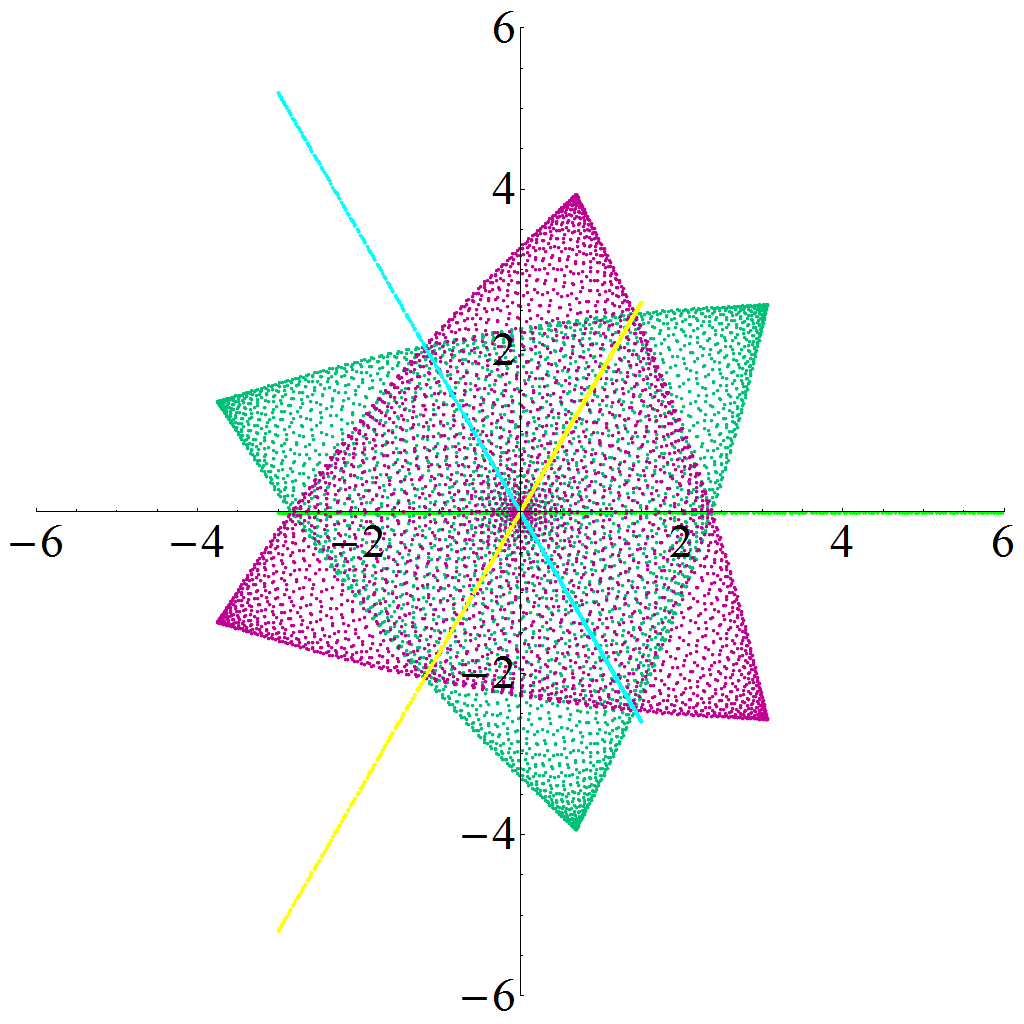}
	                \caption{{\scriptsize $n=68913$, $A=\langle 88\rangle$}}
	                \label{fig:68913_88}
	        \end{subfigure}
			\quad
	        \begin{subfigure}[b]{0.30\textwidth}
	                \centering
	                \includegraphics[width=\textwidth]{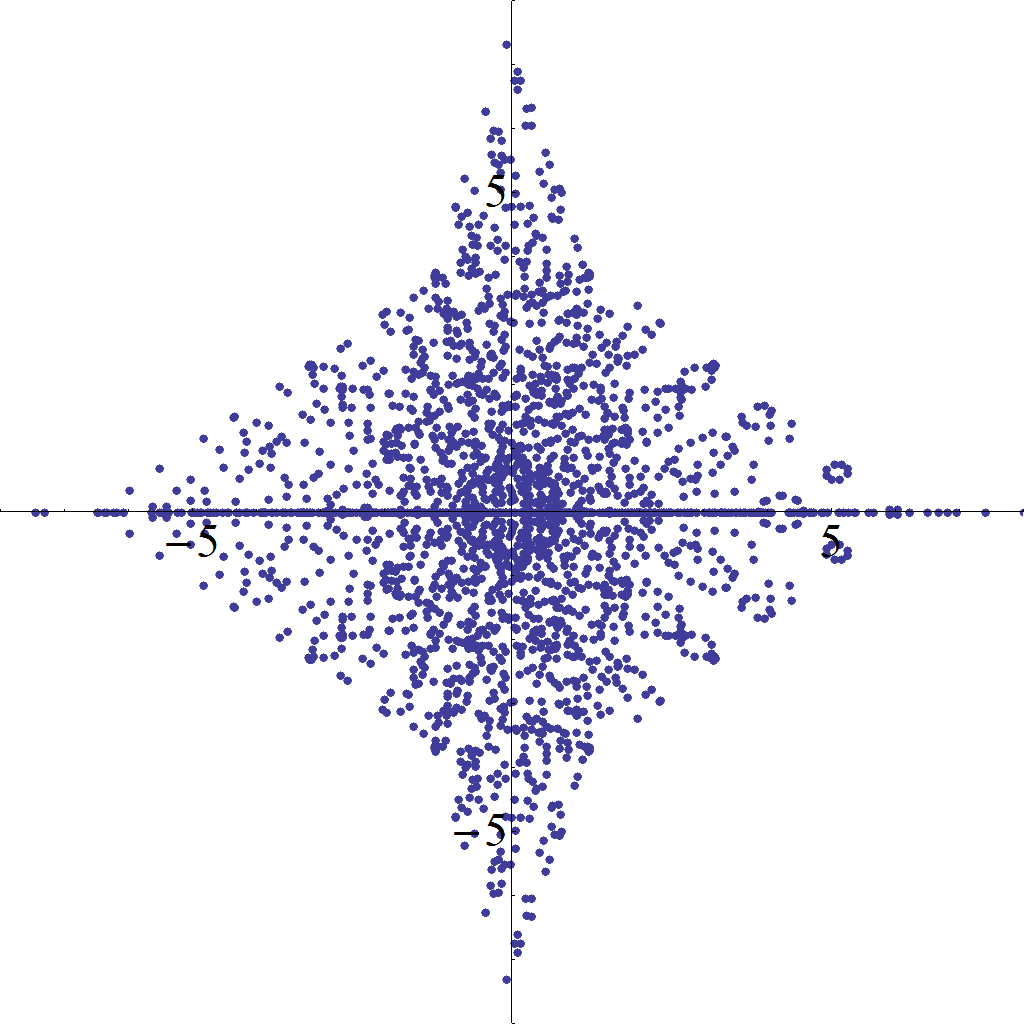}
	                \caption{{\scriptsize $n=20485$, $A=\langle 4609\rangle$}}
	                \label{fig:20485_4609}
	        \end{subfigure}
			\quad
	        \begin{subfigure}[b]{0.30\textwidth}
	                \centering
	                \includegraphics[width=\textwidth]{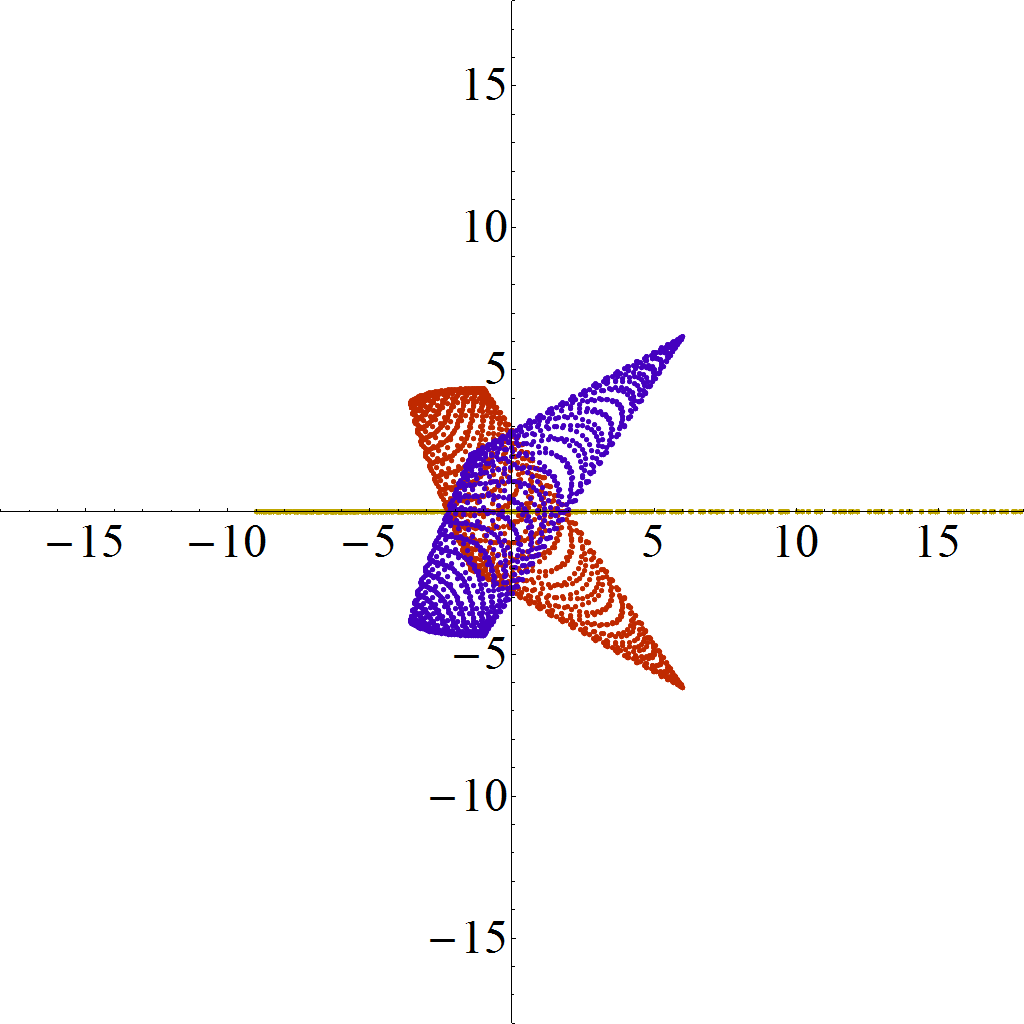}
	                \caption{{\scriptsize $n=51319$, $A=\langle 138\rangle$}}
	                \label{fig:51319_138}
	        \end{subfigure}
			\\
	        \begin{subfigure}[b]{0.30\textwidth}
	                \centering
	                \includegraphics[width=\textwidth]{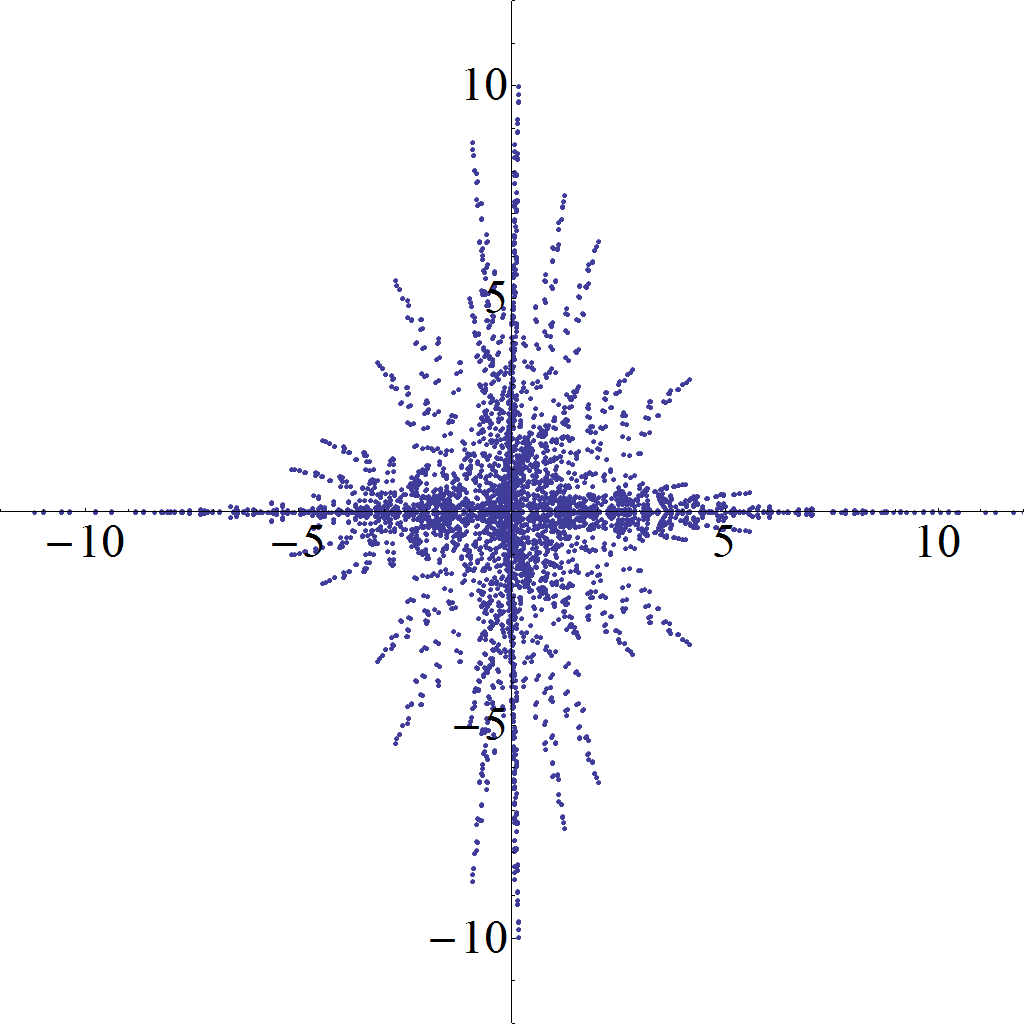}
	                \caption{{\scriptsize $n=51319$, $A=\langle 27\rangle$}}
	                \label{fig:51319_27}
	        \end{subfigure}
			\quad
	        \begin{subfigure}[b]{0.30\textwidth}
	                \centering
			\includegraphics[width=\textwidth]{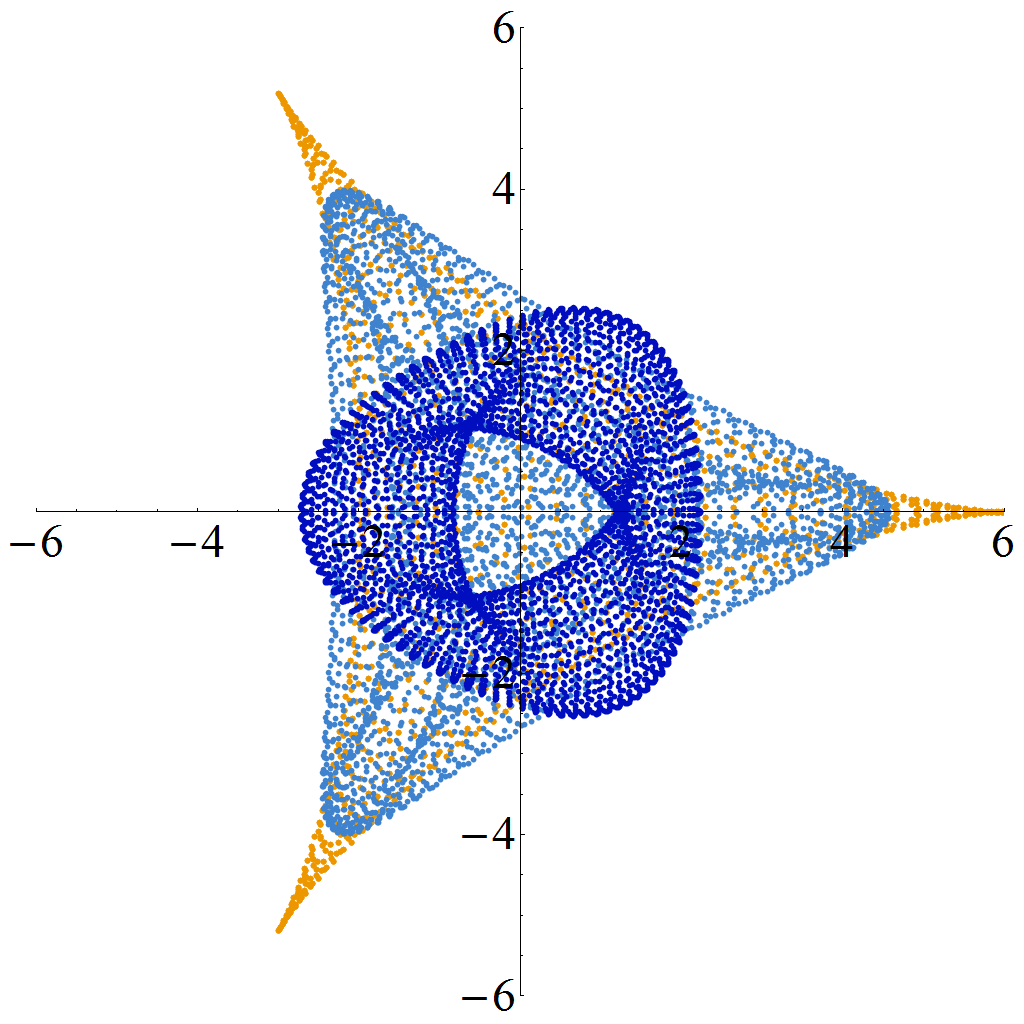}
	                \caption{{\scriptsize $n=44161$, $A=\langle 608\rangle$}}
	                \label{fig:44161_608}
	        \end{subfigure}
			\quad
	        \begin{subfigure}[b]{0.30\textwidth}
	                \centering
	                \includegraphics[width=\textwidth]{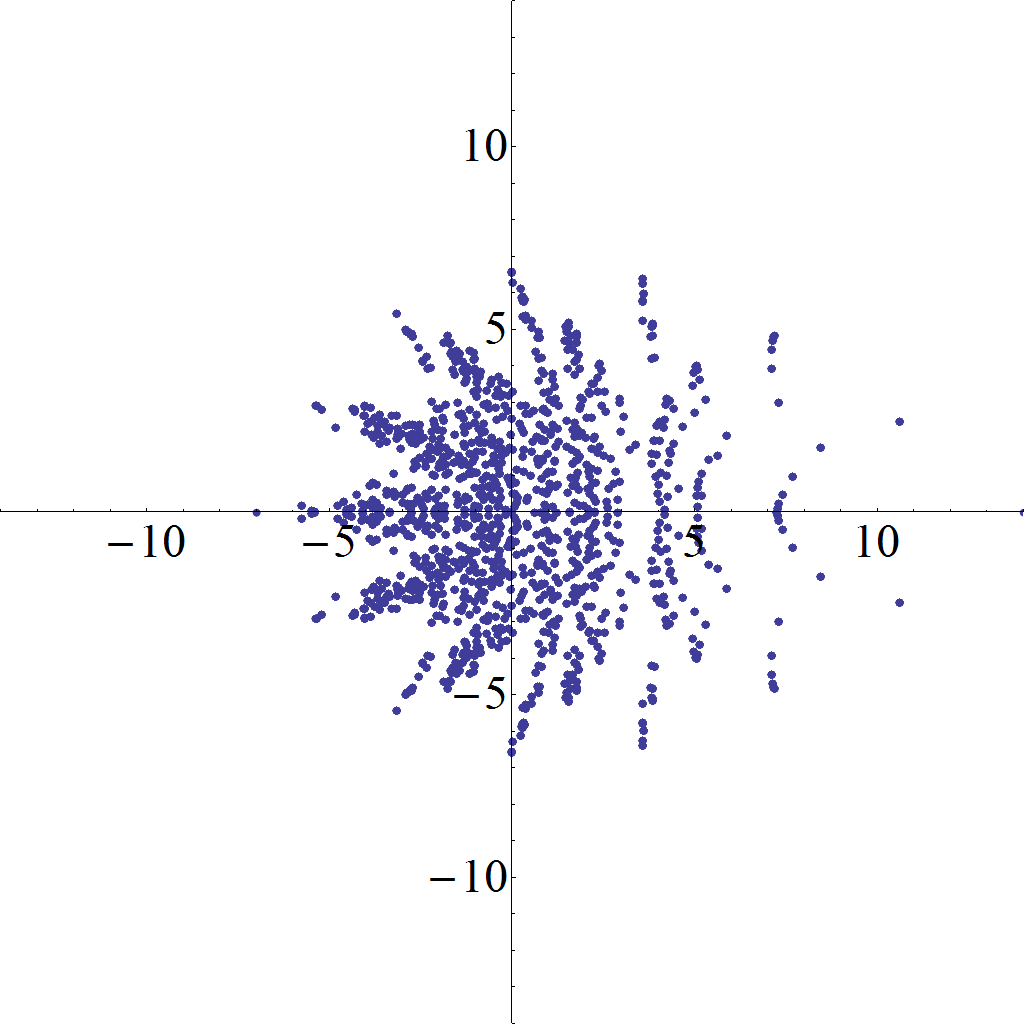}
	                \caption{{\scriptsize $N=16383$, $A=\langle 2\rangle$}}
	                \label{fig:16383_2}
	        \end{subfigure}
	        \caption{Graphs of $\sigma_X$ of $\Z/n\Z$, where $X=Ar$, fixing $r=1$. Odd values of $n/(r,n)$ can produce asymmetric images. See Figure \ref{fig:nicepics1} for a brief discussion of colorization.}
	\label{fig:axial}
	\end{figure}

	\begin{prop}\label{thm:kfold}
		If $\sigma_X$ is a cyclic supercharacter of $\Z/n\Z$, where $X=\langle \omega\rangle r$, then $\sigma_X$ has $(\omega-1,\frac{n}{(r,n)})$-fold dihedral symmetry.
	\end{prop}
	
	\begin{proof}
		Let $d=n/(r,n)$. If $k=(\omega-1,d)$, then the generator $\omega$, and hence every element of $\langle \omega\rangle$, has the form $jk+1$. 
		Since $r=\xi n/d$ for some unit $\xi$, each $x$ in $X$ has the form $(\xi n/d) (jk+1)$.
		If $y'=y+d/k$, then $y'-y-d/k\equiv 0\pmod{n}$, in which case
		\begin{equation*}
			\frac{\xi n}{d}\left(jk+1\right)\left(y'-y-\frac{d}{k}\right)\equiv 0\pmod{n}.
		\end{equation*}
		It follows that
		\begin{equation*}
			(jk+1)\left(\frac{\xi n}{d}\left(y'-y\right)-\frac{\xi n}{k}\right)\equiv 0\pmod{n},
		\end{equation*}
		whence
		\begin{align*}
			\frac{\xi n}{d}(jk+1)y'&\equiv \frac{\xi n}{d}(jk+1)y + \frac{\xi n}{k}(jk+1)\pmod{n}\\
			&\equiv \frac{\xi n}{d}(jk+1)y +\frac{\xi n}{k} \pmod{n},
		\end{align*}
		Since the function $e$ is periodic with period $1$, we have
		\begin{equation*}
			\sum_{x\in X} e\left(\frac{xy'}{n}\right)=\sum_{x\in X} e\left(\frac{xy+\xi n/k}{n}\right)=e\left(\frac{\xi}{k}\right)\sum_{x\in X} e\left(\frac{xy}{n}\right).
		\end{equation*}
		In other words, the image of $\sigma_X$ is invariant under counterclockwise rotation by $2\pi\xi/k$ about the origin. 
		If $m\xi\equiv1\pmod{k}$, then the graph is also invariant under counterclockwise rotation by
		$m\cdot2\pi\xi/k=2\pi/k$. Dihedral symmetry follows, since for all $y$ in $\Z/n\Z$, the the image of $\sigma_X$ contains both $\sigma_X(y)$ and $\overline{\sigma_X(y)}=\sigma_X(-y)$.
	\end{proof}

	\begin{example}
		For $m=1,2,3,4,6,8,12$, let $X_m$ denote the orbit of $1$ under the action of $\langle 4609\rangle$ on $\Z/(20485m)\Z$. 
		Consider the cyclic supercharacter $\sigma_{X_1}$, whose graph appears in Figure \ref{fig:20485_4609}. 
		We have $(20485,4608)=(5\cdot17\cdot 241,2^9\cdot 3^2)=1$, 
		so Theorem \ref{thm:kfold} guarantees that $\sigma_{X_1}$ has 1-fold dihedral symmetry. 
		It is visibly apparent that $\sigma_X$ has \emph{only} the trivial rotational symmetry. 
		
		Figures \ref{fig:r1} to \ref{fig:r6} display the graphs of $\sigma_{X_m}$ in the cases $m\neq 1$. 
		For each such $m$, the graph of $\sigma_{X_m}$ contains a scaled copy of $\sigma_{X_1}$ 
		by Theorem \ref{thm:master} and has $m$-fold dihedral symmetry by Theorem \ref{thm:kfold}, 
		since $(20485m,4608)=m$. It is evident from the associated figures 
		that $m$ is maximal in each case, in the sense that $\sigma_{X_m}$ having $k$-fold dihedral symmetry implies $k\leq m$.
	\end{example}

		\begin{figure}[h]
		\centering
		        \begin{subfigure}[b]{0.30\textwidth}
		                \centering
				\includegraphics[width=\textwidth]{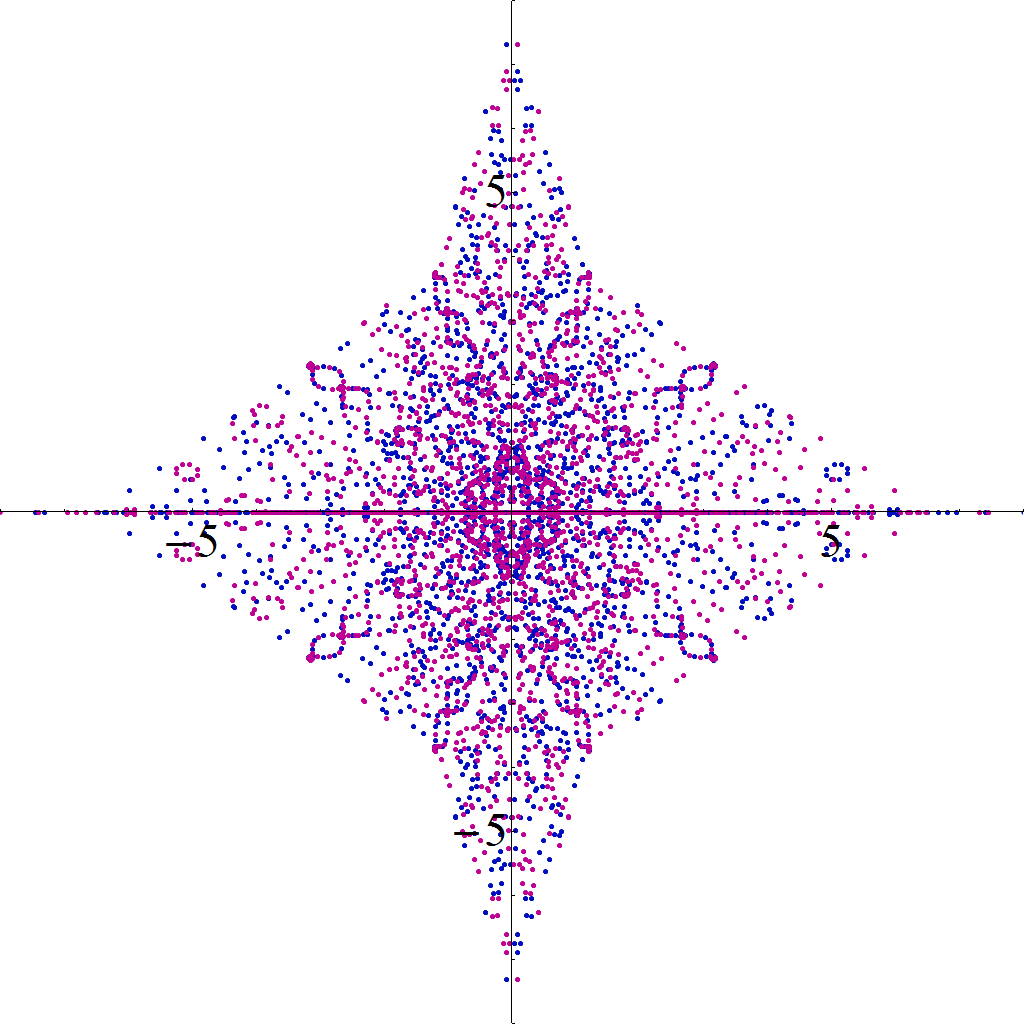}
		                \caption{{\scriptsize $n=2\cdot 20485$}}
		                \label{fig:r1}
		        \end{subfigure}
				\quad
		        \begin{subfigure}[b]{0.30\textwidth}
		                \centering
		                \includegraphics[width=\textwidth]{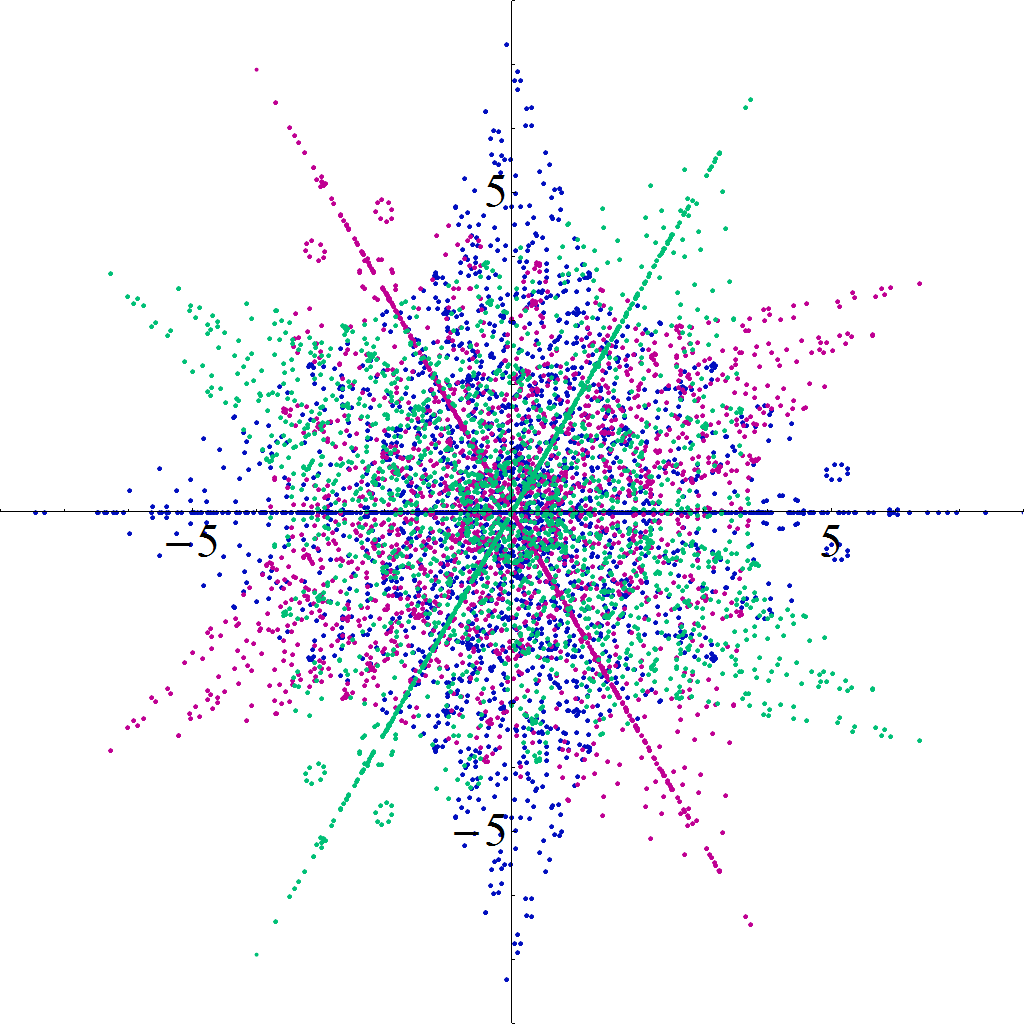}
		                \caption{{\scriptsize $n=3\cdot 20485$}}
		                \label{fig:r2}
		        \end{subfigure}
				\quad
				\begin{subfigure}[b]{0.30\textwidth}
		                \centering
		                \includegraphics[width=\textwidth]{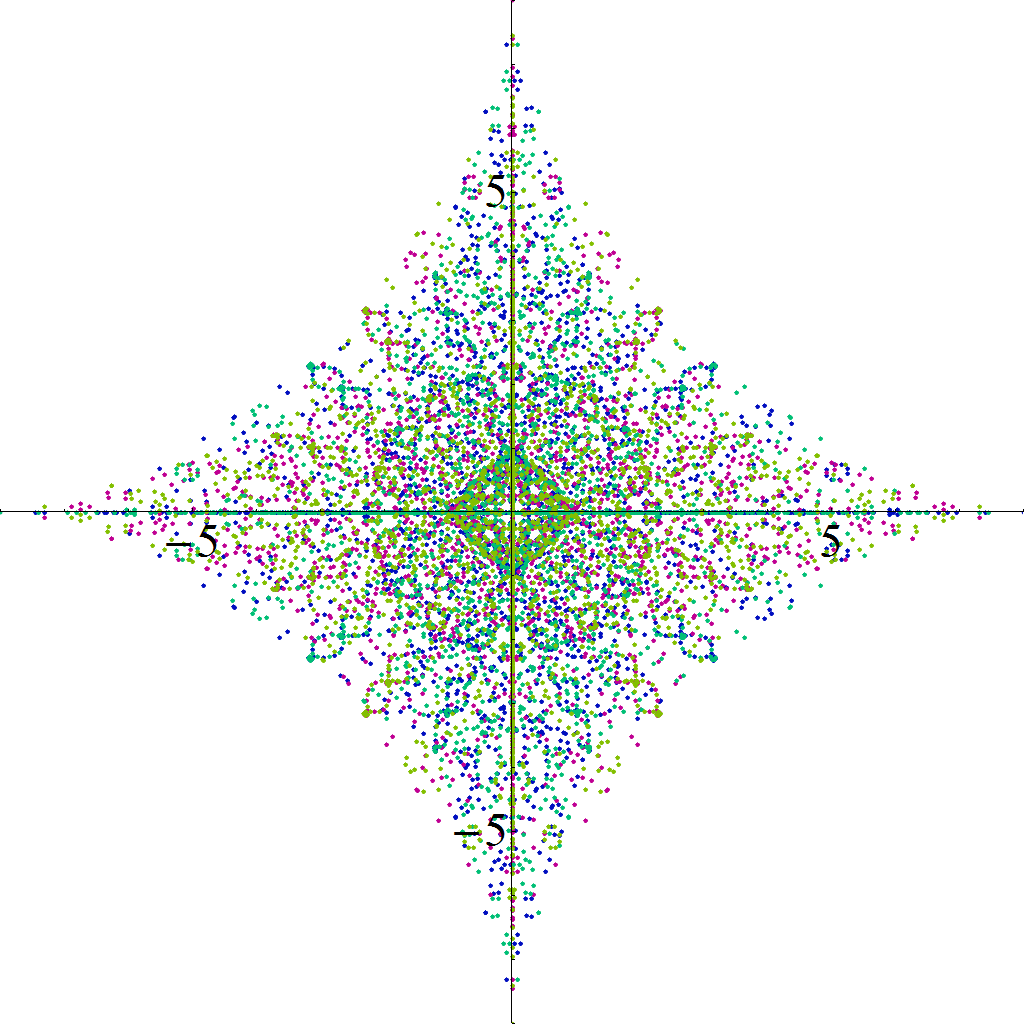}
		                \caption{{\scriptsize $n=4\cdot 20485$}}
		                \label{fig:r3}
		        \end{subfigure}
				\\
		        \begin{subfigure}[b]{0.30\textwidth}
		                \centering
		                \includegraphics[width=\textwidth]{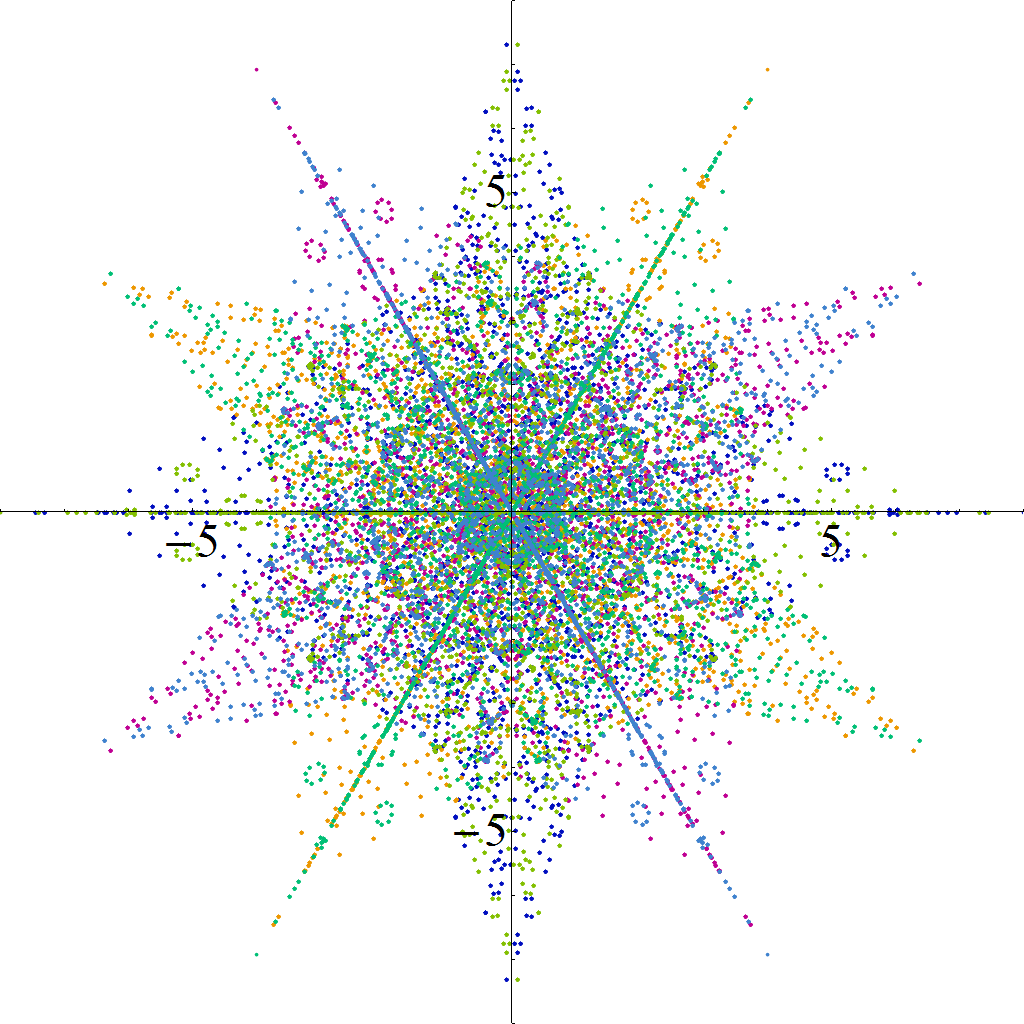}
		                \caption{{\scriptsize $n=6\cdot 20485$}}
		                \label{fig:r4}
		        \end{subfigure}
				\quad
		        \begin{subfigure}[b]{0.30\textwidth}
		                \centering
		                \includegraphics[width=\textwidth]{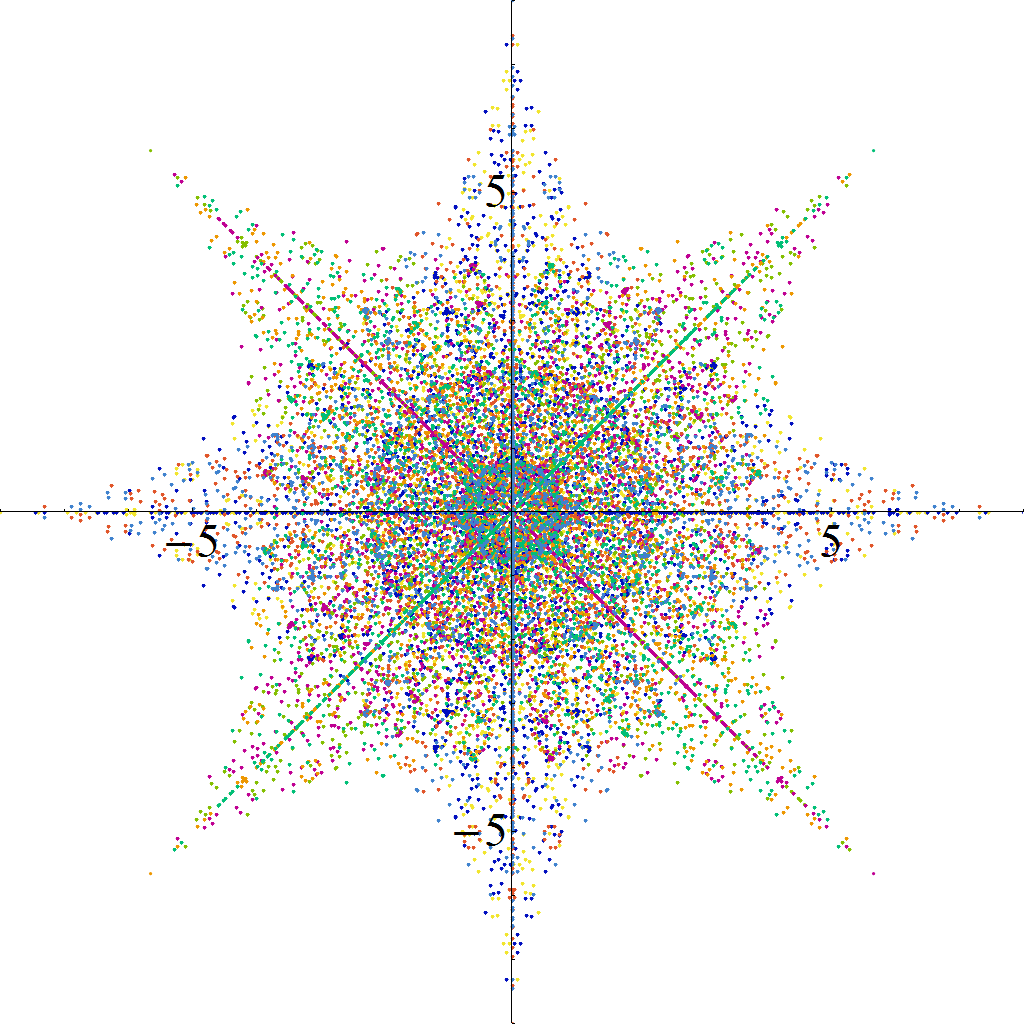}
		                \caption{{\scriptsize $n=8\cdot 20485$}}
		                \label{fig:r5}
		        \end{subfigure}
				\quad
				\begin{subfigure}[b]{0.30\textwidth}
		                \centering
		                \includegraphics[width=\textwidth]{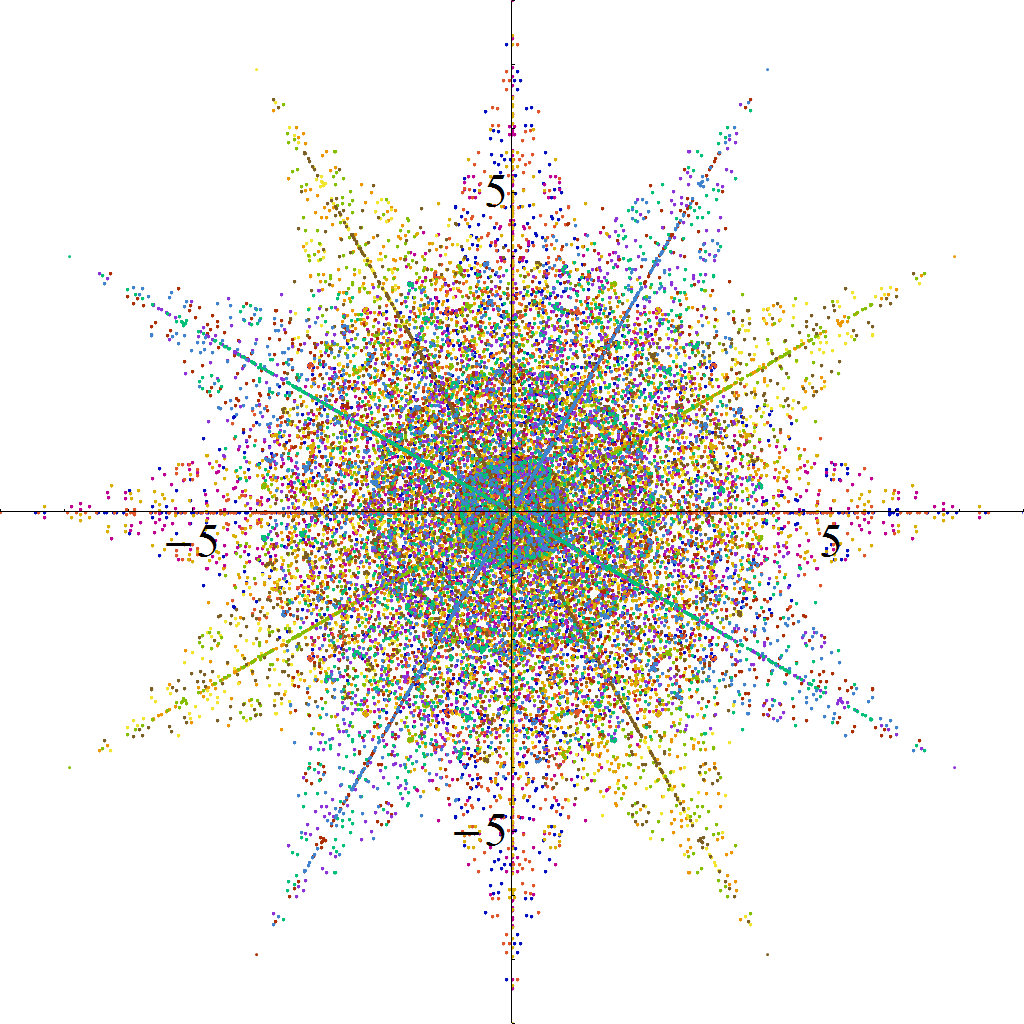}
		                \caption{{\scriptsize $n=12\cdot 20485$}}
		                \label{fig:r6}
		        \end{subfigure}
		        \caption{Graphs of cyclic supercharacters $\sigma_X$ of $\Z/n\Z$, where $X=\langle 4609\rangle 1$. By taking multiples of $n$, 
		        we produce dihedrally symmetric images containing the one in Figure \ref{fig:20485_4609}, each rotated copy of which is colored differently.}
		\label{fig:rotational}
		\end{figure}

\section{Real and imaginary supercharacters}\label{SectionReal}
	The images of some cyclic supercharacters are subsets of the real axis. 
	Many others are subsets of the union of the real and imaginary axes. 
	In this section, we establish sufficient conditions for each situation to occur and provide explicit evaluations in certain cases. Let $\sigma_X$ be a cyclic supercharacter of $\Z/n\Z$, where $X=Ar$. If $A$ contains $-1$, then it is immediate from \eqref{eq-sigma} that $\sigma_X$ is real-valued.

	\begin{example}
		Let $X$ be the orbit of $3$ under the action of $\langle 164\rangle$ on $\Z/855\Z$. 
		Since $164^3\equiv -1\pmod{n}$, it follows that $\sigma_X$ is real-valued, as suggested by
		Figure \ref{fig:855_164}.
	\end{example}

	\begin{example}
		If $A=\langle -1\rangle$ and $X=Ar$ where $r\neq \tfrac{n}{2}$, then $X=\{-r,r\}$ and
		$\sigma_X(y)=2\cos(2\pi ry/n)$.  Figure \ref{fig:105_104} illustrates this situation.
	\end{example}

	\begin{figure}[h]
		\centering
	        \begin{subfigure}[b]{0.30\textwidth}
	                \centering
	                \includegraphics[width=\textwidth]{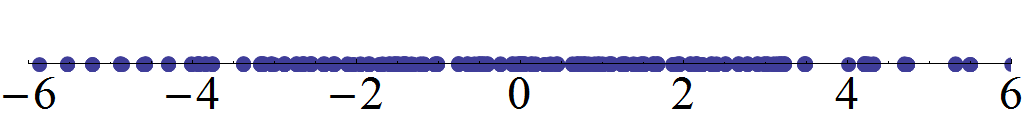}
	                \caption{{\scriptsize $n=855$, $A=\langle 164\rangle$}}
	                \label{fig:855_164}
	        \end{subfigure}
			\quad
                \begin{subfigure}[b]{0.30\textwidth}
	                \centering
	                \includegraphics[width=\textwidth]{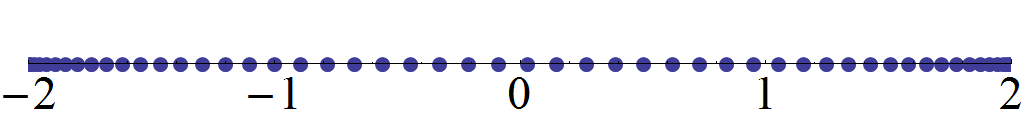}
	                \caption{{\scriptsize $n=105$, $A=\langle 104\rangle$}}
	                \label{fig:105_104}
	        \end{subfigure}
	        \quad
	        \begin{subfigure}[b]{0.30\textwidth}
	                \centering
	                \includegraphics[width=\textwidth]{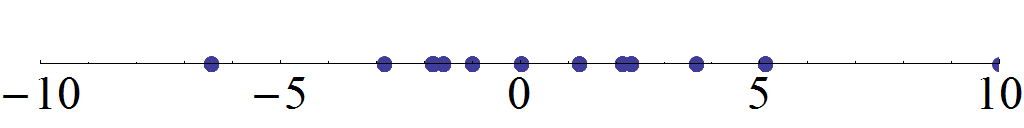}
	                \caption{{\scriptsize $n=121$, $A=\langle 94\rangle$}}
	                \label{fig:605_578}
	        \end{subfigure}
	        \caption{Graphs of cyclic supercharacters $\sigma_X$ of $\Z/n\Z$, where $X=A1$.
	        Each $\sigma_X$ is real-valued, since each $A$ contains $-1$.}
		\label{fig:realval}
	\end{figure}

	We turn our attention to cyclic supercharacters whose values, if not real, are purely imaginary
	(see Figure \ref{fig:realim}).  To this end, we introduce the following notation.	
	Let $k$ be a positive divisor of $n$, and suppose that
	\begin{equation}\label{eq:j0}
		\qquad A=\left\langle j_0n/k-1\right\rangle,\quad\mbox{for some }1\leq j_0< k.
	\end{equation}
	In this situation, we have
	\begin{equation*}
		\left(j_0n/k-1\right)^m\equiv(-1)^m\quad\left(\bmod\,\tfrac{n}{k}\right),
	\end{equation*}
	so that every element of $A$ has either the form $\tfrac{jn}{k}+1$ or $\tfrac{jn}{k}-1$, where $0\leq j<k$. In this situation, we write
	\begin{equation}\label{eq:pmorbit}
		A=\left\{jn/k+1 : j\in J_+\right\}\cup\left\{jn/k-1 : j\in J_-\right\}
	\end{equation}
	for some subsets $J_+$ and $J_-$ of $\{0,1,\ldots,k-1\}$. 
	
The condition \eqref{eq:pmorbit} is vacuous if $k=n$. 
	However, if $k<n$ and $j_0>1$ (i.e., if $A$ is nontrivial), then it follows that $(-1)^{|A|}\equiv 1\pmod{\tfrac{n}{k}}$, whence $|A|$ is even. 
	In particular, this implies $|J_+|=|J_-|$. The subsets $J_+$ and $J_-$ are not necessarily disjoint.
	For instance, if $A=\langle -1\rangle=\{-1,1\}$, then \eqref{eq:pmorbit} holds where $k=1$ and $J_+=J_-=\{0\}$. 
	In general, $J_+$ must contain 0, since $A$ must contain $1$. The following result is typical of those obtainable by imposing restrictions on $J_+$ and $J_-$.

	\begin{prop}\label{improp}
		Let $\sigma_X$ be a cyclic supercharacter of $\Z/n\Z$, where $X=Ar$, and suppose that \eqref{eq:pmorbit} holds, 
		where $k$ is even and $J_-=\tfrac{k}{2}-J_+$. 
		\begin{enumerate}
			\item[(i)] If $r$ is even, then the image of $\sigma_X$ is a subset of the real axis.
			\item[(ii)] If $r$ is odd, then $\sigma_X(y)$ is real whenever $y$ is even and purely imaginary whenever $y$ is odd.
		\end{enumerate}
	\end{prop}

	\begin{proof}
		Each $x$ in $X$ has the form $(jn/k+1)r$ or $\left(\left(k/2-j\right)n/k+1\right)r$. If $y=2m$ for some integer $m$, then for every summand $e(xy/n)$ in the definition of $\sigma_X(y)$ having the form $e\left(2m(jn/k+1)r/n\right)$,
		there is one of the form $e\left(2m(n/2-jn/k+1)r/n\right)$, its complex conjugate. From this we deduce that $\sigma_X(y)$ is real whenever $y$ is even. If $y=2m+1$, then for every summand of the form $e\left((2m+1)\left(jn/k+1\right)r/n\right)$, there is one of the form $e\left((2m+1)(n/2-jn/k+1)r/n\right)$. If $r$ is odd, then the latter is the former reflected across the imaginary axis, in which case $\sigma_X(y)$ is purely imaginary. 
		If $r$ is even, then the latter is the complex conjugate of the former, in which case $\sigma_X(y)$ is real.
	\end{proof}

	\begin{figure}[h]
		\centering
	        \begin{subfigure}[b]{0.30\textwidth}
	                \centering
	                \includegraphics[width=\textwidth]{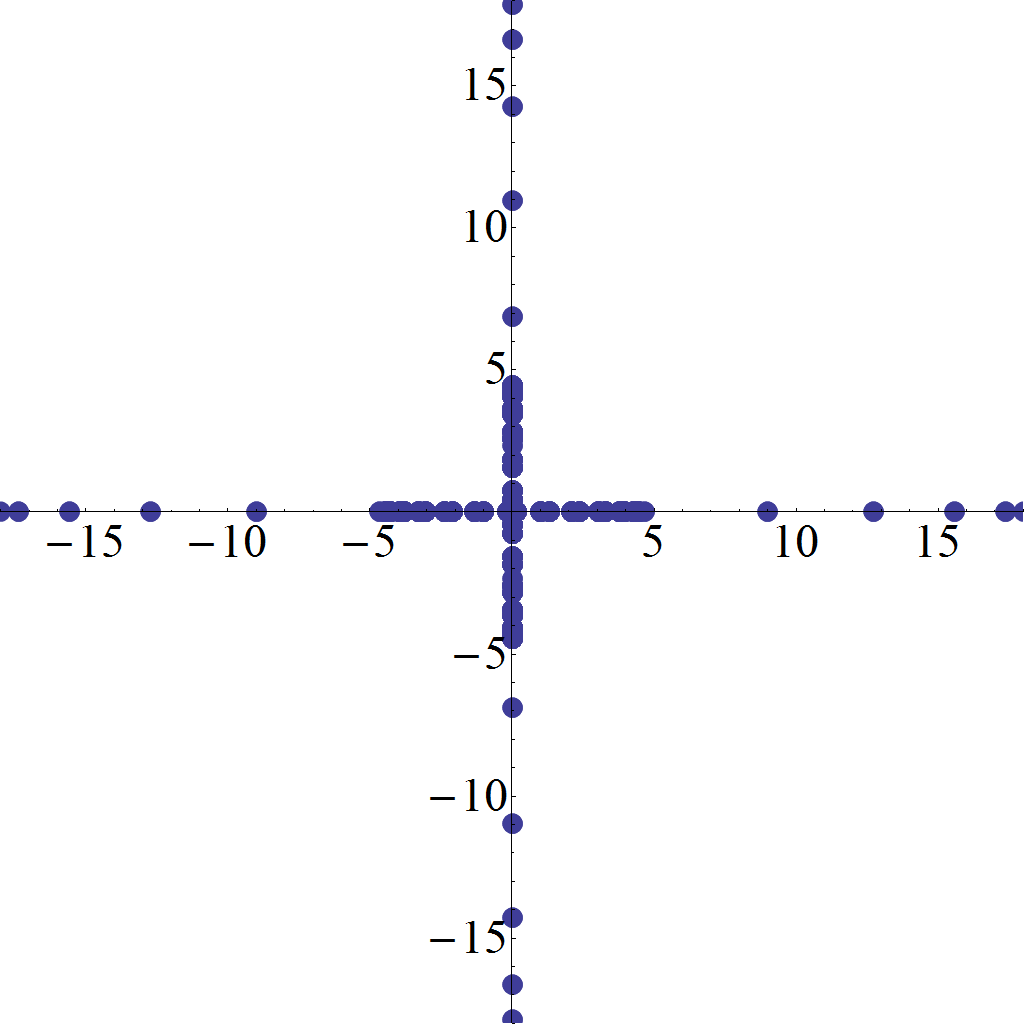}
	                \caption{{\scriptsize $n=912$, $A=\langle 71\rangle$}}
	                \label{fig:912_71}
	        \end{subfigure}
	        \quad
	        \begin{subfigure}[b]{0.30\textwidth}
	                \centering
	                \includegraphics[width=\textwidth]{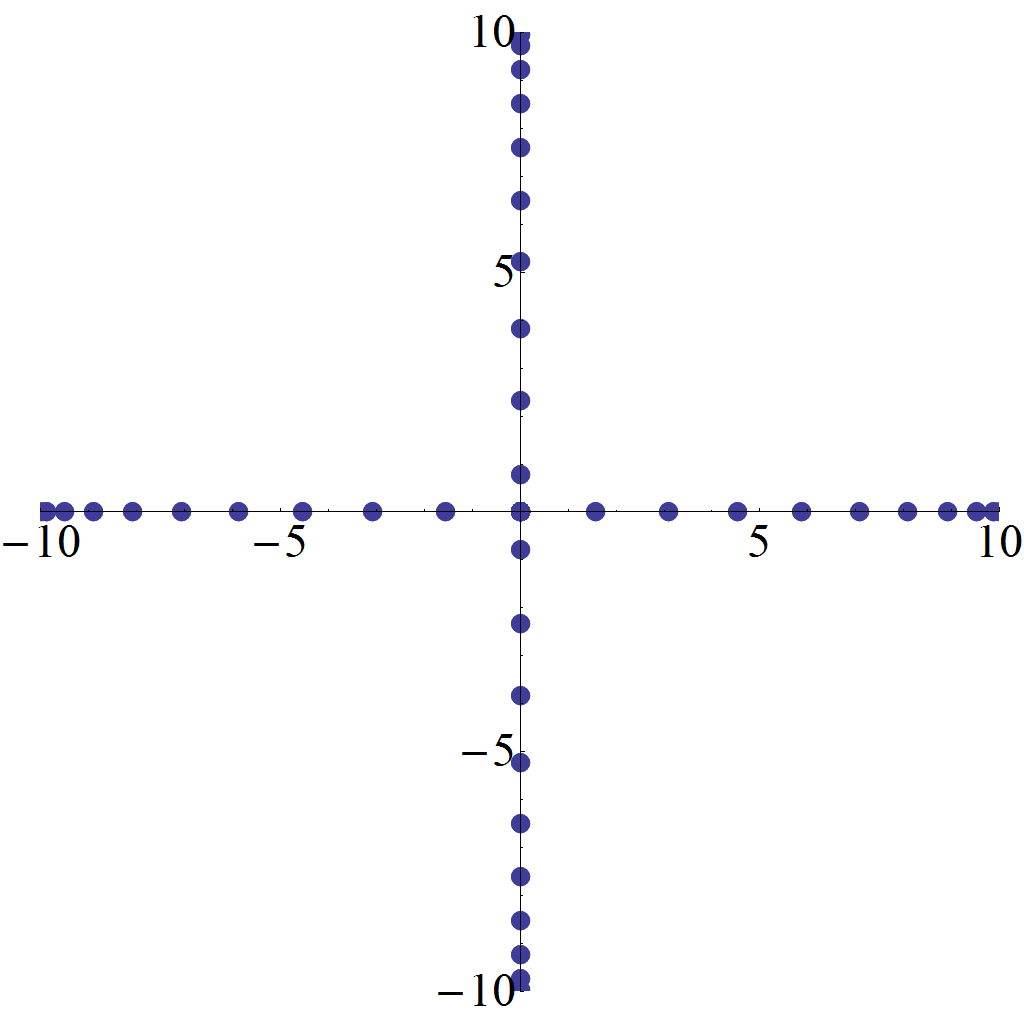}
	                \caption{{\scriptsize $n=400$, $A=\langle 39\rangle$}}
	                \label{fig:400_39}
	        \end{subfigure}
		\quad
	        \begin{subfigure}[b]{0.30\textwidth}
	                \centering
	                \includegraphics[width=\textwidth]{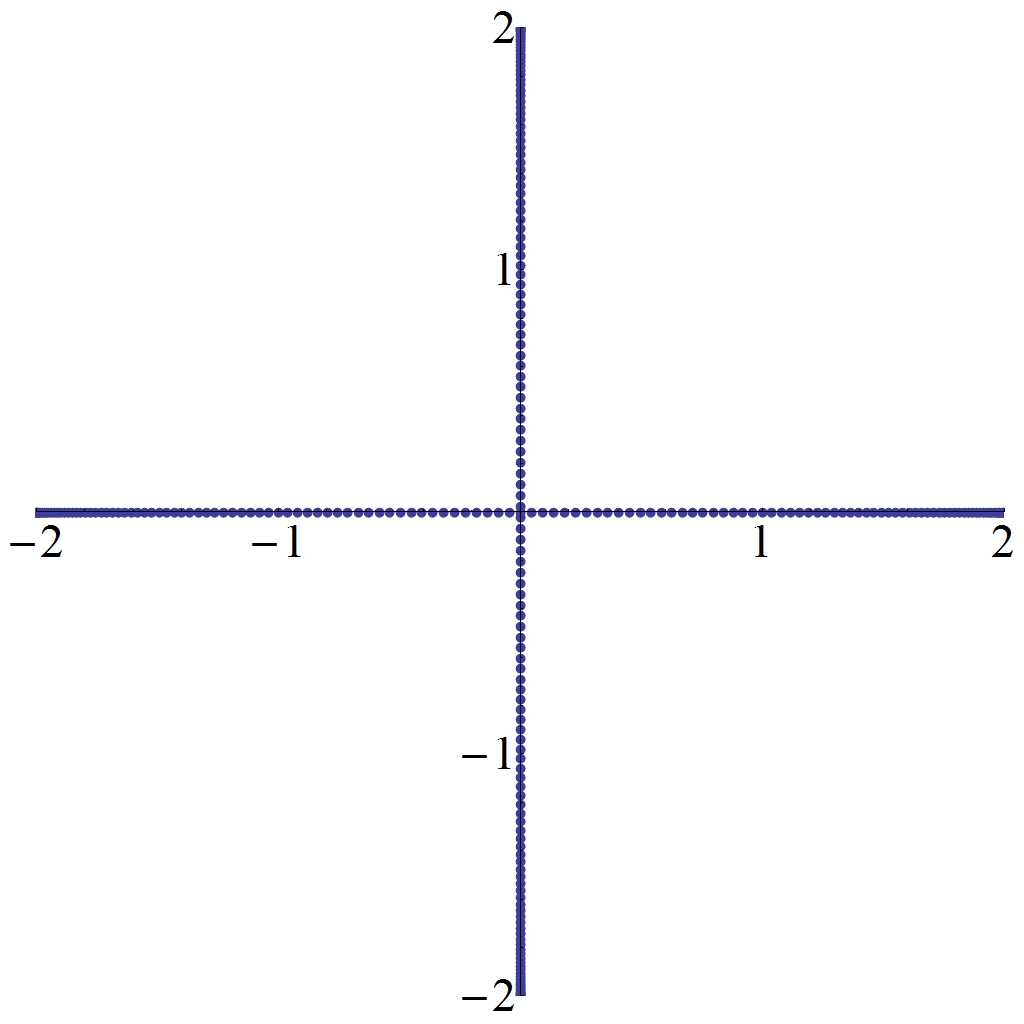}
	                \caption{{\scriptsize $n=552$, $A=\langle 275\rangle$}}
	                \label{fig:552_275}
	        \end{subfigure}
	        \caption{Graphs of cyclic supercharacters $\sigma_X$ of $\Z/n\Z$, where $X=A1$. Some cyclic supercharacters have values that are either real or purely imaginary.}
		\label{fig:realim}
	\end{figure}
	
	\begin{example}
		In the case of Figure \ref{fig:912_71}, we have $n=912$, $r=1$, $k=38$, $j_0=3$,
		$$J_+=\{0,2,12,16,20,22,24,26,32\},\quad\mbox{and}\quad
		J_-=\{3,7,17,19,25,31,33,35,37\},$$
		so the hypotheses of Proposition \ref{improp}(ii) hold.
	\end{example}
	
	An explicit evaluation of $\sigma_X$ is available if 
	$J_+\cup J_-=\{0,1,\ldots,k-1\}$. The following result, presented without proof, treats this situation 
	(see Figure \ref{fig:400_39}).

	\begin{prop}\label{explicitprop}
		Suppose that $k>2$ is even, and that \eqref{eq:pmorbit} holds where $J_+$ is the set of all even residues modulo $k$ 
		and $J_-$ is the set of all odd residues. If $X$ is the orbit of a unit $r$ under the action of $A$ on $\Z/n\Z$, then
		\begin{equation*}
			\sigma_X(y)=
			\begin{cases}
				k\cos\frac{2\pi ry}{n}&\mbox{if } k|y,\\
				ik\sin\frac{2\pi ry}{n}&\mbox{if }y\equiv \frac{k}{2}\pmod{k},\\
				0&\mbox{otherwise}.\\
			\end{cases}
		\end{equation*}
	\end{prop}

\section{Ellipses}

	Discretized ellipses appear frequently in the graphs of cyclic supercharacters.
	These, in turn, form primitive elements whence more complicated supercharacter plots emerge.
	In order to proceed, we recall the definition of a Gauss sum.
	Suppose that $m$ and $k$ are integers with $k>0$. If $\chi$ is a Dirichlet character modulo $k$, 
	then the \emph{Gauss sum} associated with $\chi$ is given by
	\begin{equation*}
		G(m,\chi)=\sum_{\ell=1}^k \chi(\ell)e\left(\frac{\ell m}{k}\right).
	\end{equation*}
	If $p$ is prime, the \emph{quadratic Gauss sum} $g(m;p)$ over $\Z/p\Z$ is given by 
	$g(m;p)=g(m,\chi)$, where $\chi(a)=\legendre{a}{p}$ is the Legendre symbol of $a$ and $p$. That is,
	\begin{equation*}
		g(m;p)=\sum_{\ell=0}^{k-1}e\left(\frac{m\ell^2}{p}\right).
	\end{equation*}
	We require the following well-known result \cite[Thm.~1.5.2]{berndt}.

	\begin{lem}
		If $p\equiv 1\pmod{4}$ is prime and $(m,p)=1$, then
		\begin{equation*}
			g(m;p)=\legendret{m}{p}\sqrt{p}.
			\label{thm:berndt}
		\end{equation*}
	\end{lem}

	\begin{prop}\label{thm:ellipse}
		Suppose that $p| n$ and $p\equiv 1\pmod{4}$ is prime.  Let
		\begin{equation*}
			Q_p=\{m\in \Z/p\Z : \legendret{m}{p}=1\}
		\end{equation*}
		denote the set of distinct nonzero quadratic residues modulo $p$. 
		If \eqref{eq:pmorbit} holds where
		\begin{equation}\label{eq:ellipsej}
			J_+=\{aq+b : q\in Q_p\}\quad\mbox{and}\quad J_-=\{cq-b : q\in Q_p\}
		\end{equation}
		for integers $a,b,c$ coprime to $p$ with $\legendre{a}{p}=-\legendre{c}{p}$, then $\sigma_X(y)$ belongs to the real interval 
		$[1-p,p-1]$ whenever $p|y$, and otherwise belongs to the ellipse described by the equation $(\Re z)^2+(\Im z)^2/p=1$.
	\end{prop}

	\begin{proof}
		For all $y$ in $\Z/n\Z$, we have
		\begin{align*}
			\sigma_X(y)&=\sum_{x\in A} e\left(\frac{xy}{n}\right)\\
			&=\sum_{j\in J_+}e\left(\frac{\left(\frac{jn}{p}+1\right)y}{n}\right)+\sum_{j\in J_-}e\left(\frac{(\frac{jn}{p}-1)y}{n}\right)\\
			&=\sum_{q\in Q_p}e\left(\frac{(aq+b)y}{p}+\frac{y}{n}\right)+\sum_{q\in Q_p}e\left(\frac{(cq-b)y}{p}-\frac{y}{n}\right)\\
			&=e\left(\frac{by}{p}+\frac{y}{n}\right)\sum_{q\in Q_p} e\left(\frac{aqy}{p}\right)
				+e\left(-\frac{by}{p}-\frac{y}{n}\right)\sum_{q\in Q_p} e\left(\frac{cqy}{p}\right)\\
			&=e(\theta_y)\sum_{\ell=1}^{(p-1)/2} e\left(\frac{a\ell^2y}{p}\right)+\overline{e(\theta_y)}
				\sum_{\ell=1}^{(p-1)/2} e\left(\frac{c\ell^2y}{p}\right),
		\end{align*}
		where $\theta_y=\tfrac{(bn+p)y}{pn}$. If $p|y$, then $e(\theta_y)=e(\frac{y}{n})$ and $e(\frac{a\ell^2y}{p})=e(\frac{c\ell^2y}{p})=1$, so
		\begin{equation*}
			\sigma_X(y)=\frac{(p-1)}{2}\left(e\left(\frac{y}{n}\right)+\overline{e\left(\frac{y}{n}\right)}\right)=(p-1)\cos\frac{2\pi y}{n}.
		\end{equation*}
		If not, then $(p,y)=1$, so
		\begin{align}\label{eq:ellipse}
			\sigma_X(y)&=\frac{e\left(\theta_y\right)\left(g(ay;p)-1\right)+\overline{e(\theta_y)}(g(cy;p)-1)}{2}\nonumber\\
			&=\frac{e(\theta_y)g(ay;p)+\overline{e(\theta_y)}g(cy;p)}{2}-\cos2\pi\theta_y\nonumber\\
			&=\frac{\sqrt{p}}{2}\left(\legendret{ay}{p}e(\theta_y)+\legendret{cy}{p}\overline{e(\theta_y)}\right)-\cos2\pi\theta_y\\
			&=\pm\legendret{y}{p}\frac{\sqrt{p}}{2}\left(e(\theta_y)-\overline{e(\theta_y)}\right)-\cos2\pi\theta_y\nonumber\\
			&=\pm i\legendret{y}{p}\sqrt{p}\sin2\pi\theta_y-\cos2\pi\theta_y,\nonumber
		\end{align}
		where \eqref{eq:ellipse} follows from Lemma \ref{thm:berndt}.
	\end{proof}
	
	\begin{figure}[h]
		\centering
	        \begin{subfigure}[b]{0.30\textwidth}
	                \centering
	                \includegraphics[width=\textwidth]{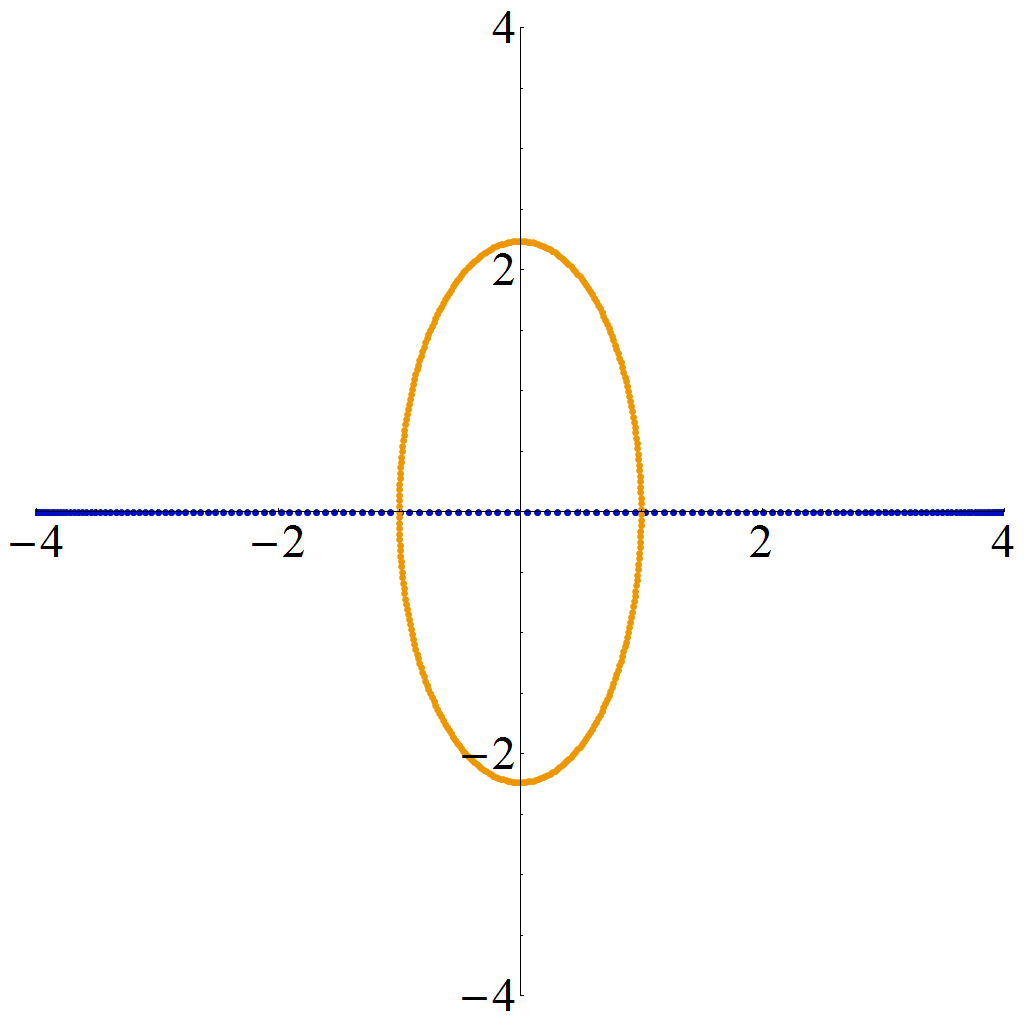}
	                \caption{{\scriptsize $n=1535$, $A=\langle 613\rangle$}}
	                \label{fig:1535_613}
	        \end{subfigure}
			\quad
	                \begin{subfigure}[b]{0.30\textwidth}
	                \centering
	                \includegraphics[width=\textwidth]{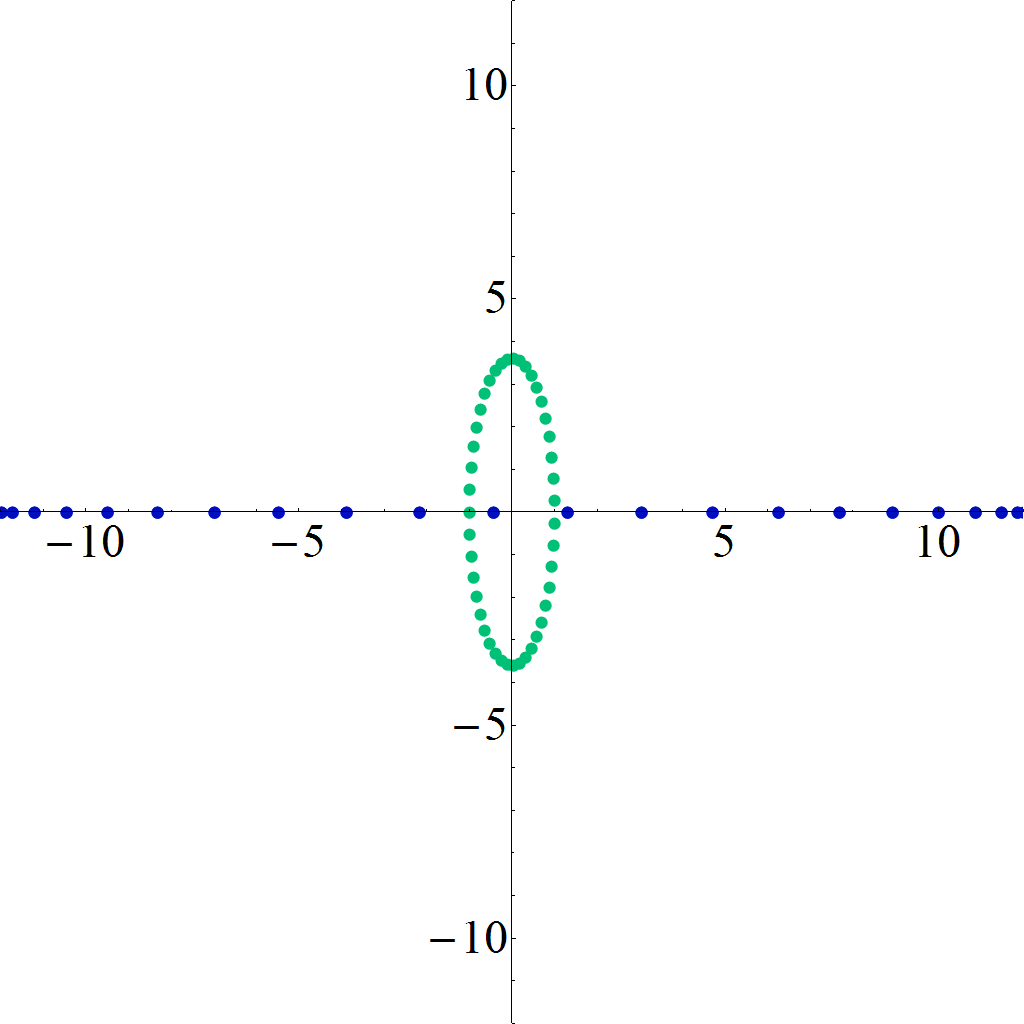}
	                \caption{{\scriptsize $n=559$, $A=\langle 171\rangle$}}
	                \label{fig:559_171}
	        \end{subfigure}
	        \quad
	                \begin{subfigure}[b]{0.30\textwidth}
	                \centering
	                \includegraphics[width=\textwidth]{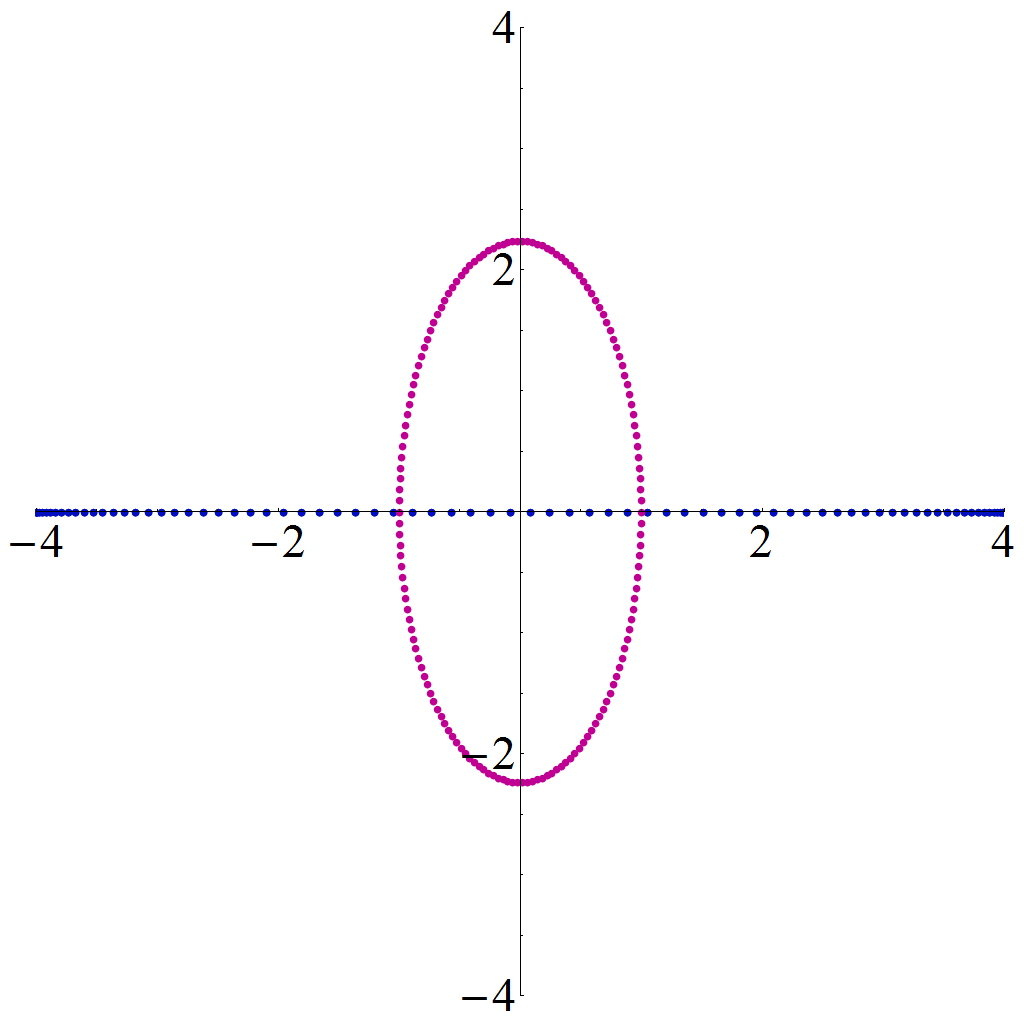}
	                \caption{{\scriptsize $n=770$, $A=\langle 153\rangle$}}
	                \label{fig:770_153}
	        \end{subfigure}
			\\
	        \begin{subfigure}[b]{0.30\textwidth}
	                \centering
	                \includegraphics[width=\textwidth]{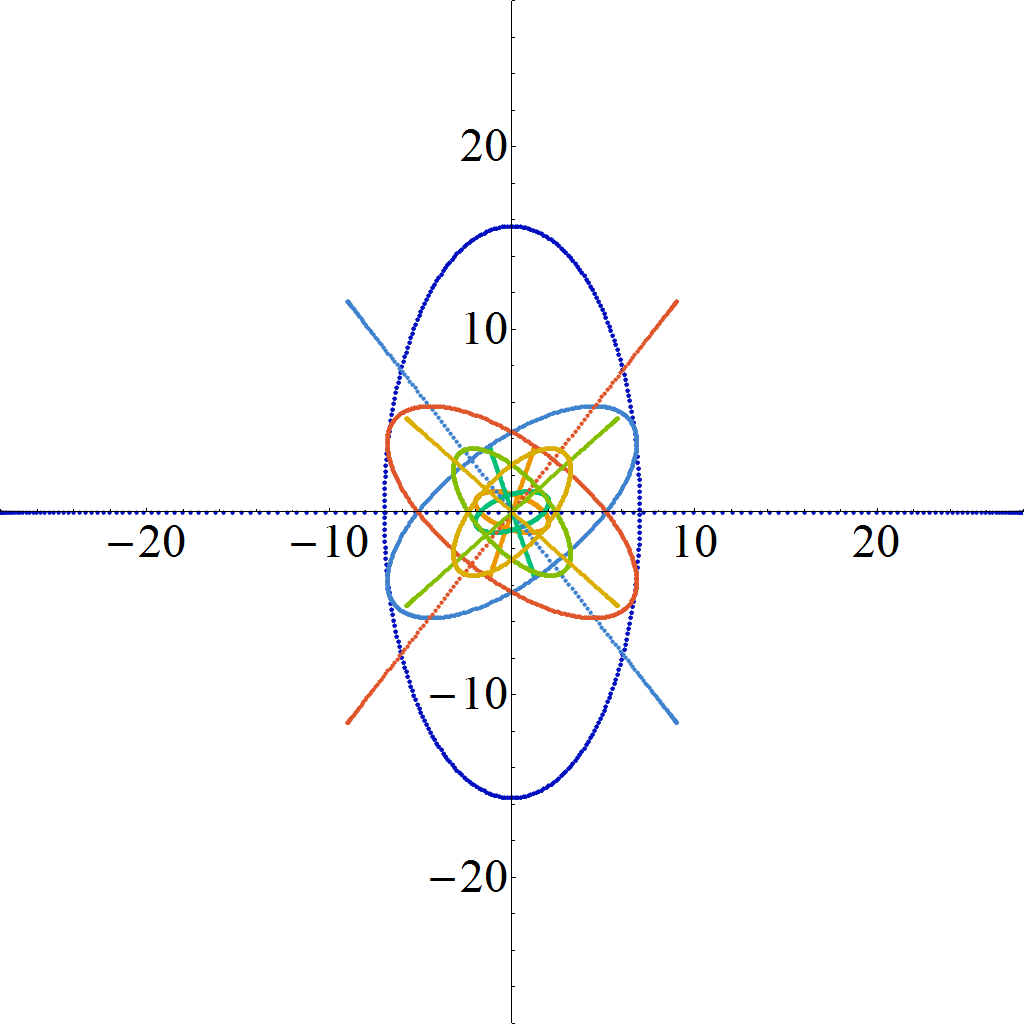}
	                \caption{{\scriptsize $n=1535\cdot 43$, $A=\langle 613\rangle$}}
	                \label{fig:1535_43}
	        \end{subfigure}
	        \quad
	        \begin{subfigure}[b]{0.30\textwidth}
	                \centering
	                \includegraphics[width=\textwidth]{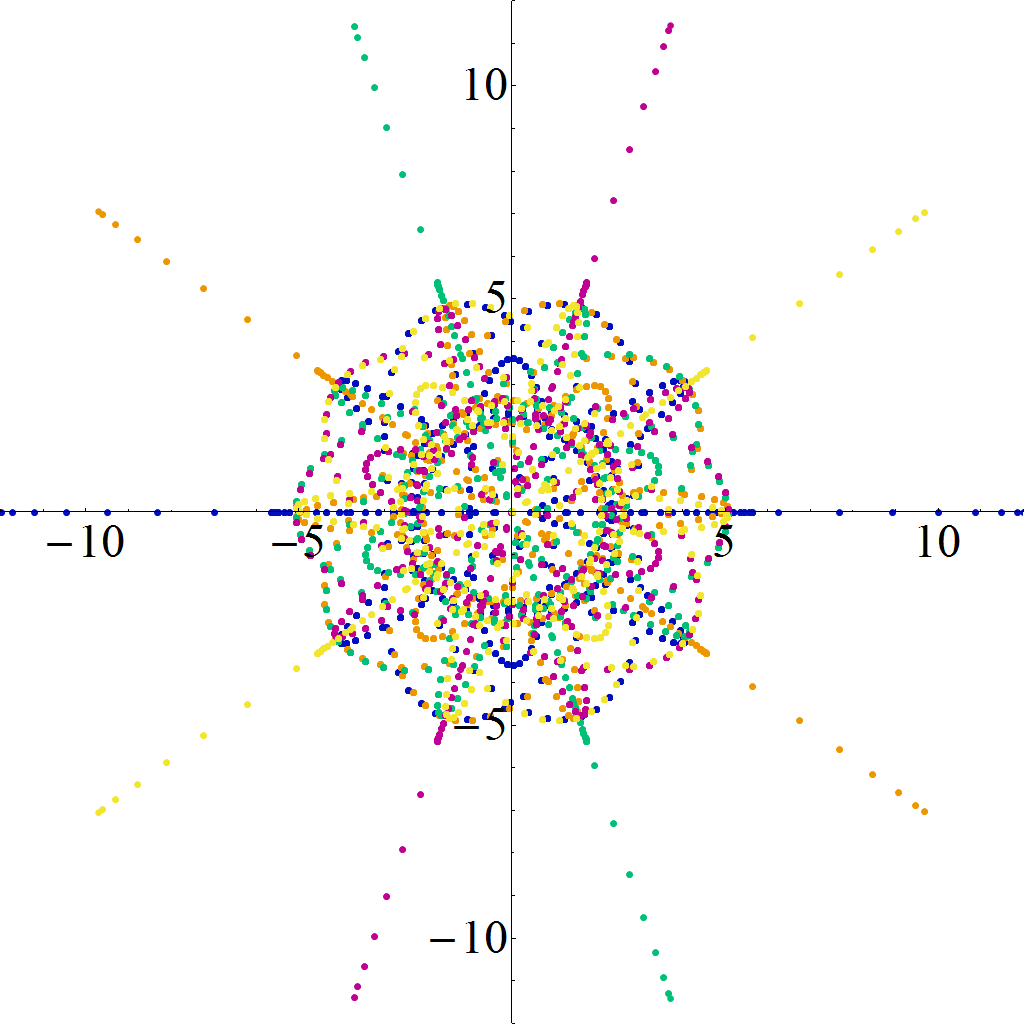}
	                \caption{{\scriptsize $n=559\cdot 7\cdot 5$, $A=\langle 171\rangle$}}
	                \label{fig:559_35}
	        \end{subfigure}
		\quad
	        \begin{subfigure}[b]{0.30\textwidth}
	                \centering
	                \includegraphics[width=\textwidth]{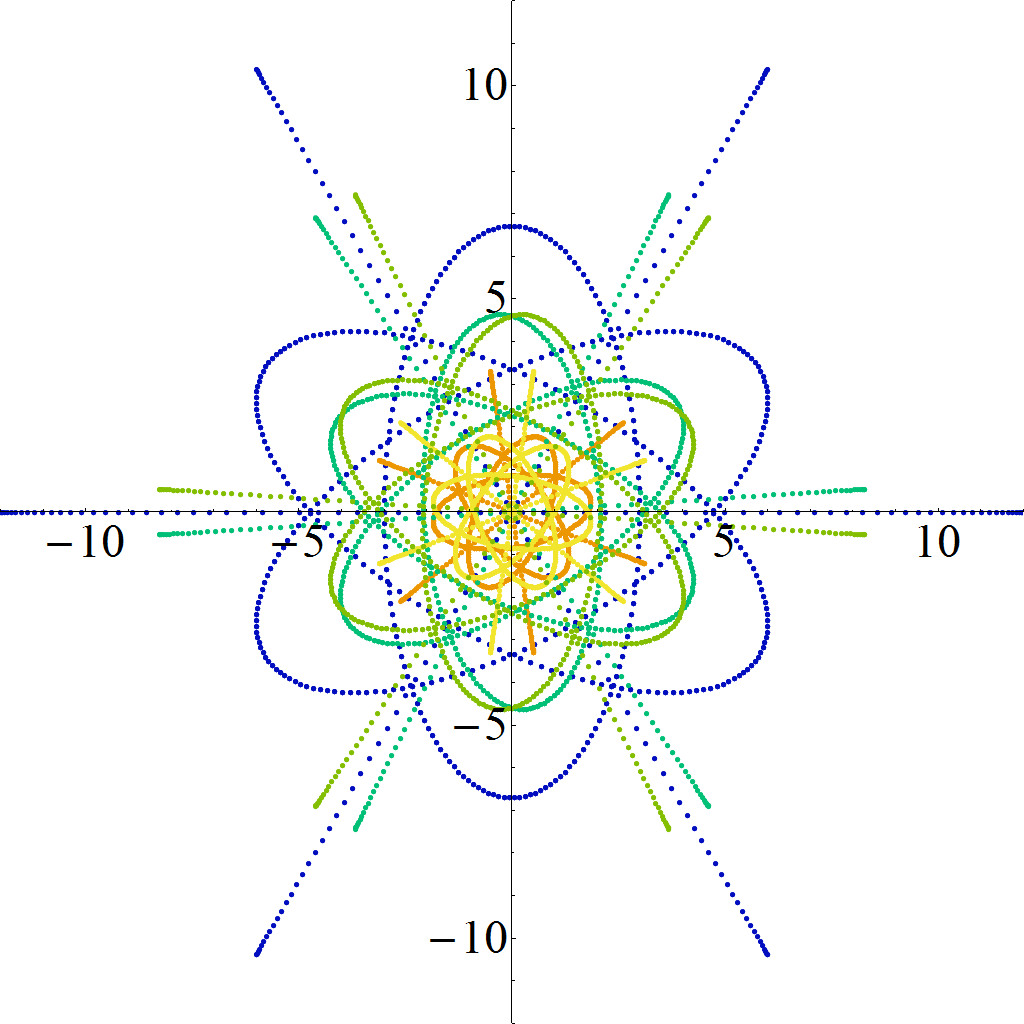}
	                \caption{{\scriptsize $n=770\cdot13\cdot3$, $A=\langle 1693\rangle$}}
	                \label{fig:770_39}
	        \end{subfigure}
	        \caption{Graphs of cyclic supercharacters $\sigma_X$ of $\Z/n\Z$, where $X=A1$. Propositions \ref{thm:master}, \ref{thm:kfold} and \ref{thm:ellipse} can be used to produce supercharacters 
	        whose images feature elliptical patterns. See Figure \ref{fig:nicepics1} for a brief discussion of colorization.}
		\label{fig:ellipses}
	\end{figure}

	\begin{example}\label{eg:singell}
		Let $n=d=1088=4^3\cdot 17$ and consider the orbit $X$ of $r=1$ under the action of 
		$A=\langle 63\rangle=\langle \frac{n}{17}-1\rangle$ on $\Z/n\Z$. In this situation, 
		illustrated by Figure \ref{fig:1535_613}, $\eqref{eq:pmorbit}$ holds with $J_+=\{0,4\}=2Q_{5}+2$ and $J_-=\{2,4\}=Q_{17}+3$. Figure \ref{fig:559_171} 
		illustrates the situation $J_+=Q_{13}+3$ and $J_-=2Q_{13}-3$, while Figure \ref{fig:770_153} 
		illustrates $J_+=Q_5+1$ and $J_-=2Q_2-1$. The remainder of Figure \ref{fig:ellipses} 
		demonstrates the effect of using Propositions \ref{thm:master}, \ref{thm:kfold} and \ref{thm:ellipse} 
		to produce supercharacters whose images feature ellipses.
	\end{example}

\section{Asymptotic behavior}\label{SectionAsymptotic}
	We now turn our attention to an entirely different matter, namely the asymptotic behavior of cyclic supercharacter plots.
	To this end we begin by recalling several definitions and results concerning uniform distribution modulo $1$.	
	The \emph{discrepancy} of a finite subset $S$ of $[0,1)^m$ is the quantity
	\begin{equation*}
		D(S) = \sup_B \left| \frac{ |B\cap S|}{|S|} - \mu(B)\right|,
	\end{equation*}
	where the supremum runs over all boxes $B = [a_1,b_1)\times\cdots \times [a_m,b_m)$ 
	and $\mu$ denotes $m$-dimensional Lebesgue measure.  We say that 
	a sequence $S_n$ of finite subsets of $[0,1)^d$ is \emph{uniformly distributed} 
	if $\lim_{n\to\infty} D(S_n) = 0$.  If $S_n$ is a sequence of finite subsets in $\R^m$,
	we say that $S_n$ is \emph{uniformly distributed mod $1$} if the corresponding
	sequence of sets $\big\{  ( \{x_1\},\{x_2\},\ldots,\{x_d\}) : (x_1,x_2,\ldots,x_m) \in S_n\big\}$
	is uniformly distributed in $[0,1)^m$.  Here $\{x\}$ denotes the fractional part $x-\lfloor x \rfloor$
	of a real number $x$.  The following fundamental result is due to H.~Weyl \cite{weyl}.
	
	\begin{lem}
		A sequence of finite sets $S_n$ in $\R^m$ is uniformly distributed modulo $1$ if and only if
		\begin{equation*}
			\lim_{n\to\infty} \frac{1}{|S_n|} \sum_{\vec{u}\in S_n} e( \vec{u} \cdot \vec{v}) = 0
		\end{equation*}
		for each $\vec{v}$ in $\Z^m$.
	\end{lem}
	
	In the following, we suppose that $q=p^a$ is a nonzero power of an odd prime and that $|X| = d$ is a divisor of $p-1$.
	Let $\omega_q$ denote a primitive $d$th root of unity modulo $q$ and let
	\begin{equation*}
		S_q = \left\{ \frac{\ell}{q}(1,\omega_q, \omega_q^2,\dots, \omega_q^{\phi(d)-1}) : \ell=0,1,\ldots,q-1 \right\}
		\subseteq [0,1)^{\phi(d)}
	\end{equation*}
	where $\phi$ denotes the Euler totient function.
	The following lemma of Myerson, whose proof we have adapted to suit our notation, can be found in \cite[Thm.~12]{myerson}.
	
	\begin{lem}\label{myersonlemma}
		The sets $S_q$ for $q \equiv 1 \pmod{d}$ are uniformly distributed modulo $1$.
	\end{lem}
	
	\begin{proof}
		Fix a nonzero vector $\vec{v} = (a_0,a_1,\ldots,a_{\phi(d)-1})$ in $\Z^{\phi(d)}$ and let 
		\begin{equation*}
		f(t) = a_0 + a_1 t + \cdots + a_{\phi(d)-1}t^{\phi(d)-1}.
		\end{equation*}
		Let $r=q/(q,f(\omega_q))$, and observe that
		\begin{align*}
			\sum_{\vec{u} \in S_q} e( \vec{u}\cdot \vec{v})
			&= \sum_{\ell=0}^{q-1} e\left( \frac{f(\omega_q) \ell}{q} \right)\\
			&= \sum_{m=0}^{q/r-1} \sum_{\ell=mr}^{(m+1)r-1} e\left(\frac{f(\omega_q)\ell}{r}\right)\\
			&=\frac{q}{r}\sum_{\ell=0}^{r-1} e\left(\frac{f(\omega_q)\ell}{r}\right)\\
			&= 
			\begin{cases}
				q & \mbox{if } q |f(\omega_q),\\
				0 & \mbox{else}.\\
			\end{cases}
		\end{align*}
		Having fixed $d$ and $\vec{v}$, we claim that the sum above is nonzero for only finitely many $q \equiv 1 \pmod{d}$.
		Letting $\Phi_d$ denote the $d$th cyclotomic polynomial, recall that $\deg \Phi_d = \phi(d)$ and
		that $\Phi_d$ is the minimal polynomial of any primitive $d$th root of unity. 
		Clearly the gcd of $f(t)$ and $\Phi_d(t)$ as polynomials in $\Q[t]$ is in $\Z$. 
		Thus there exist $a(t)$ and $b(t)$ in $\Z[t]$ so that
		\begin{equation*}
			a(t) \Phi_d(t) + b(t) f(t) = n
		\end{equation*}
		for some integer $n$. Passing to $\Z/q\Z$ and letting $t = \omega_q$, we find that
		$b(\omega_q) f(\omega_q) \equiv n \pmod{p}$.  This means that $q|f(\omega_q)$ implies $q|n$,
		which can occur for only finitely many prime powers $q$.
		Putting this all together, we find that for all $\vec{v}$ in $\Z^{\phi(d)}$ the following holds:
		\begin{equation*}
			\lim_{ \substack{ q \to \infty \\ q \equiv 1 \!\!\!\!\!\pmod{d}}}
			\frac{1}{|S_q|} \sum_{\vec{u} \in S_q} e( \vec{u} \cdot \vec{v}) = 0.
		\end{equation*}
		By Weyl's Criterion, it follows that the sets $S_q$
		are uniformly distributed mod $1$ as $q \equiv 1 \pmod{d}$ tends to infinity.
	\end{proof}

	\begin{thm}\label{TheoremMain}
		Let $\sigma_X$ be a cyclic supercharacter of $\Z/q\Z$, where $q=p^a$ is a nonzero power of an odd prime. If $X=A1$ and $|X|=d$ divides $p-1$, 
		then the image of $\sigma_X$ is contained in the image of the function $g:[0,1)^{\phi(d)}\to\C$ defined by
		\begin{equation}\label{eq-Mapping}
			g(z_1,z_2,\ldots,z_{\phi(d)}) = \sum_{k=0}^{d-1} \prod_{j=0}^{\phi(d)-1} z_{j+1}^{b_{k,j}}
		\end{equation}
		where the integers $b_{k,j}$ are given by
		\begin{equation}\label{bkj}
			t^k\equiv \sum_{j=0}^{\phi(d)-1} b_{k,j} t^j \pmod{\Phi_d(t)}.
		\end{equation}
		For a fixed $d$, as $q$ becomes large, the image of $\sigma_X$ fills out the image of $g$,
		in the sense that, given $\epsilon>0$, there exists some $q\equiv 1 \pmod{d}$ such that if $\sigma_X:\Z/q\Z\to\C$ 
		is a cyclic supercharacter with $|X|=d$, then every open ball of radius $\epsilon > 0$ in the image of $g$ 
		has nonempty intersection with the image of $\sigma_X$.
	\end{thm}
	
	\begin{proof}
		Let $\omega_q$ be a primitive $d$th root of unity modulo $q$, so that $A=\langle \omega_q\rangle$ in $(\Z/q\Z)^{\times}$. Recall that $\{1, e(\frac{1}{d}),\dots, e(\frac{\phi(d)-1}{d}) \}$ is a $\Z$-basis for the ring of integers 
		of the cyclotomic field $\Q(e(\frac{1}{d}))$ \cite[Prop.~10.2]{neukirch}. 
		For $k=0,1,\ldots, d-1$, the integers $b_{k,j}$ in the expression
		\begin{equation*}
			e\left(\frac{k}{d}\right)= \sum_{j=0}^{\phi(d)-1} b_{k,j} e\left(\frac{j}{d}\right),
		\end{equation*}
		are determined by \eqref{bkj}. In particular, it follows that
		\begin{equation*}
			\omega_q^k\equiv \sum_{j=0}^{\phi(d)-1} b_{k,j} \omega_q^{j} \;\pmod{q}.
		\end{equation*}
		We have
		\begin{equation*}
			\sigma_X(y) 
			= \sum_{x \in X} e\left( \frac{x y}{q} \right) = \sum_{k=0}^{d-1} e\left(y \frac{\omega_q^k}{q} \right) 
			= \sum_{k=0}^{d-1} e\left(y \sum_{j=0}^{\phi(d)-1} b_{k,j} \frac{ \omega_q^{j}}{q}\right),
		\end{equation*}
		from which it follows that the image of $\sigma_X$ is contained in the image of the function
		$g:\T^{\phi(d)}\to\C$ defined by \eqref{eq-Mapping}.
		The density claim now follows immediately from Lemma \ref{myersonlemma}.
	\end{proof}

	In combination with Propositions \ref{thm:master} and \ref{primescale}, the preceding theorem characterizes the boundary curves of cyclic supercharacters with prime power moduli. If $d$ is even, then $X$ is closed under negation, so $\sigma_X$ is real. If $d=k^a$ where $k$ is an odd prime, then $g:\T^{\phi(k^a)}\to\C$ is given by
$$g(z_1,z_2,\cdots,z_{\phi(d)})=\sum_{j=1}^{\phi(d)} z_j + \sum_{j=1}^{k^{a-1}}\prod_{\ell=0}^{k-2} z^{-1}_{j+\ell k^{a-1}}.$$
A particularly concrete manifestation of our result is Theorem \ref{TheoremHypo},
	whose proof we present below.
	Recall that a \emph{hypocycloid} is a planar curve obtained by tracing the path of a distinguished point on a small circle 
	as it rolls within a larger circle. Rolling a circle of integral radius $\lambda$ within a circle of integral radius 
	$\kappa$, where $\kappa>\lambda$, yields the parametrization
	$\theta\mapsto(\kappa-\lambda)e^{i\theta}+\lambda e^{(1-\kappa/\lambda)i\theta}$
	of the hypocycloid centered at the origin that contains the point $\kappa$ and has precisely $\kappa$ cusps.

	\begin{figure}[H]
		\centering
	        \begin{subfigure}[b]{0.30\textwidth}
	                \centering
	                \includegraphics[width=\textwidth]{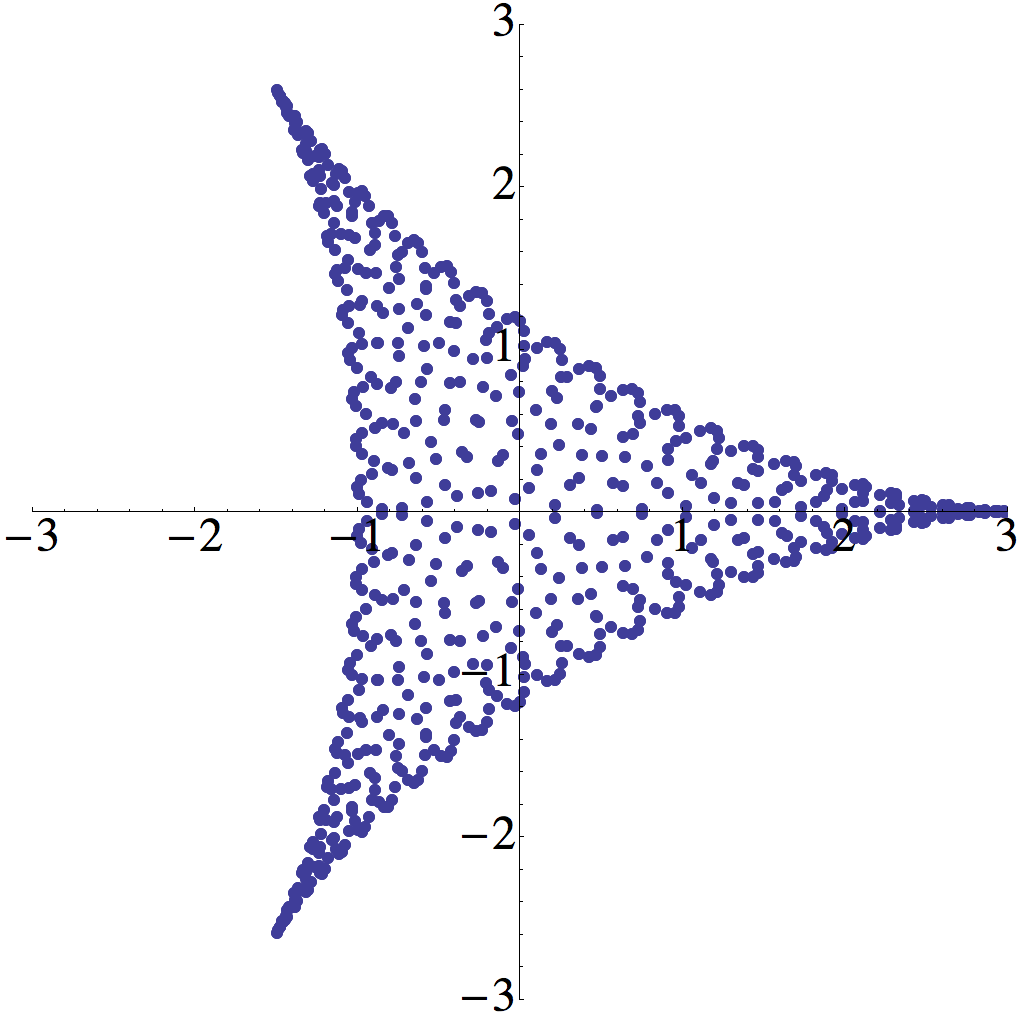}
	                \caption{{\scriptsize $p=2017$, $A=\langle 294\rangle$}}
	        \end{subfigure}
	        \quad
	                \begin{subfigure}[b]{0.30\textwidth}
	                \centering
	                \includegraphics[width=\textwidth]{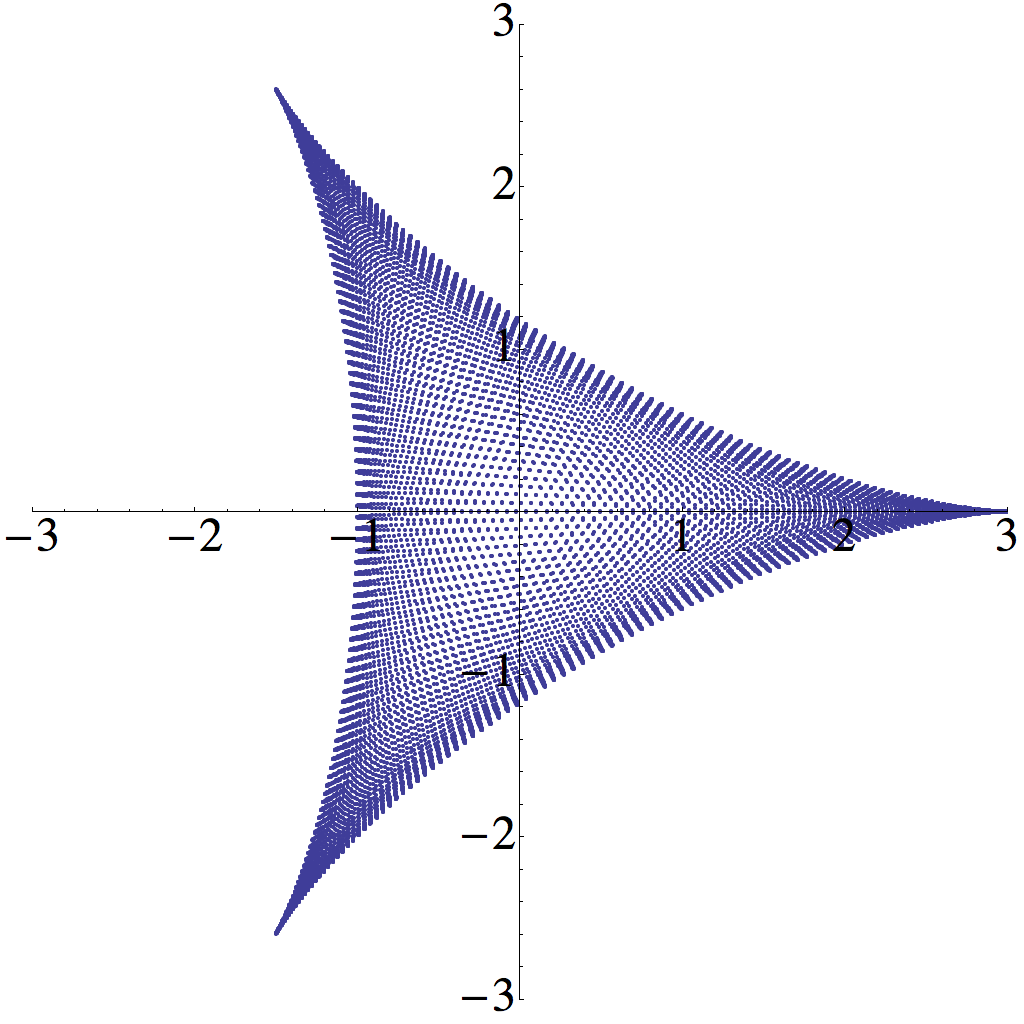}
	                \caption{{\scriptsize $p=32587$, $A=\langle 10922\rangle$}}
	        \end{subfigure}
			\quad
	        \begin{subfigure}[b]{0.30\textwidth}
	                \centering
	                \includegraphics[width=\textwidth]{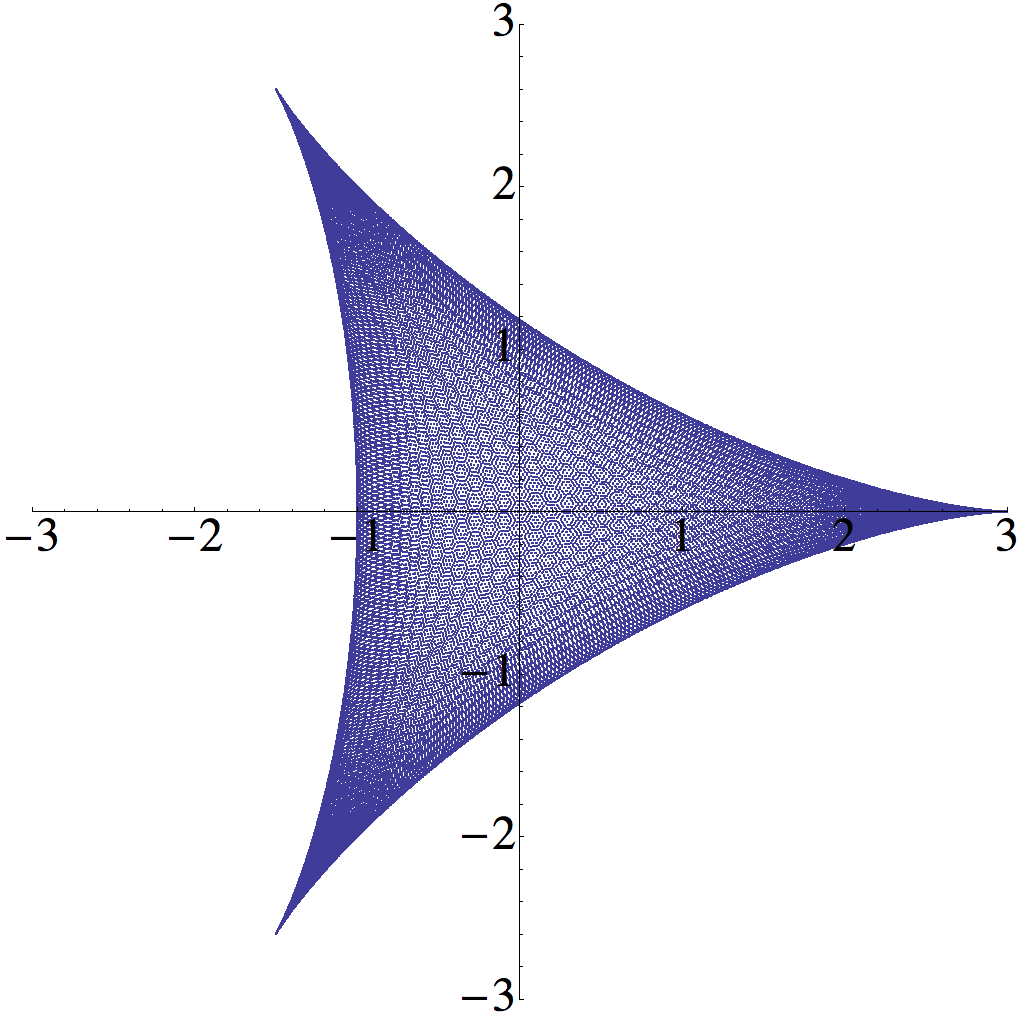}
	                \caption{{\scriptsize $p=200017$, $A=\langle 35098\rangle$}}
	        \end{subfigure}
		\\
	        \begin{subfigure}[b]{0.30\textwidth}
	                \centering
	                \includegraphics[width=\textwidth]{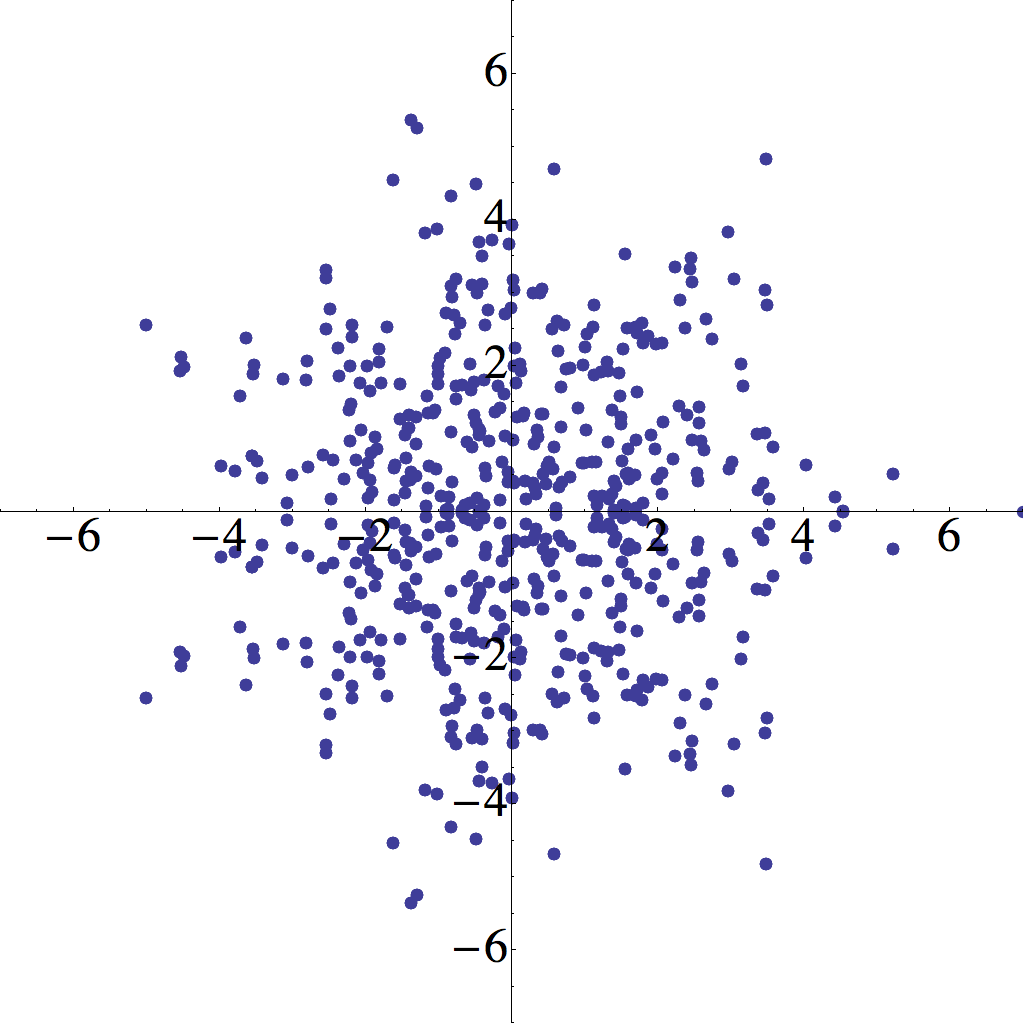}
	                \caption{{\scriptsize $p=4019$, $A=\langle 1551\rangle$}}
	        \end{subfigure}
		\quad
	        \begin{subfigure}[b]{0.30\textwidth}
	                \centering
	                \includegraphics[width=\textwidth]{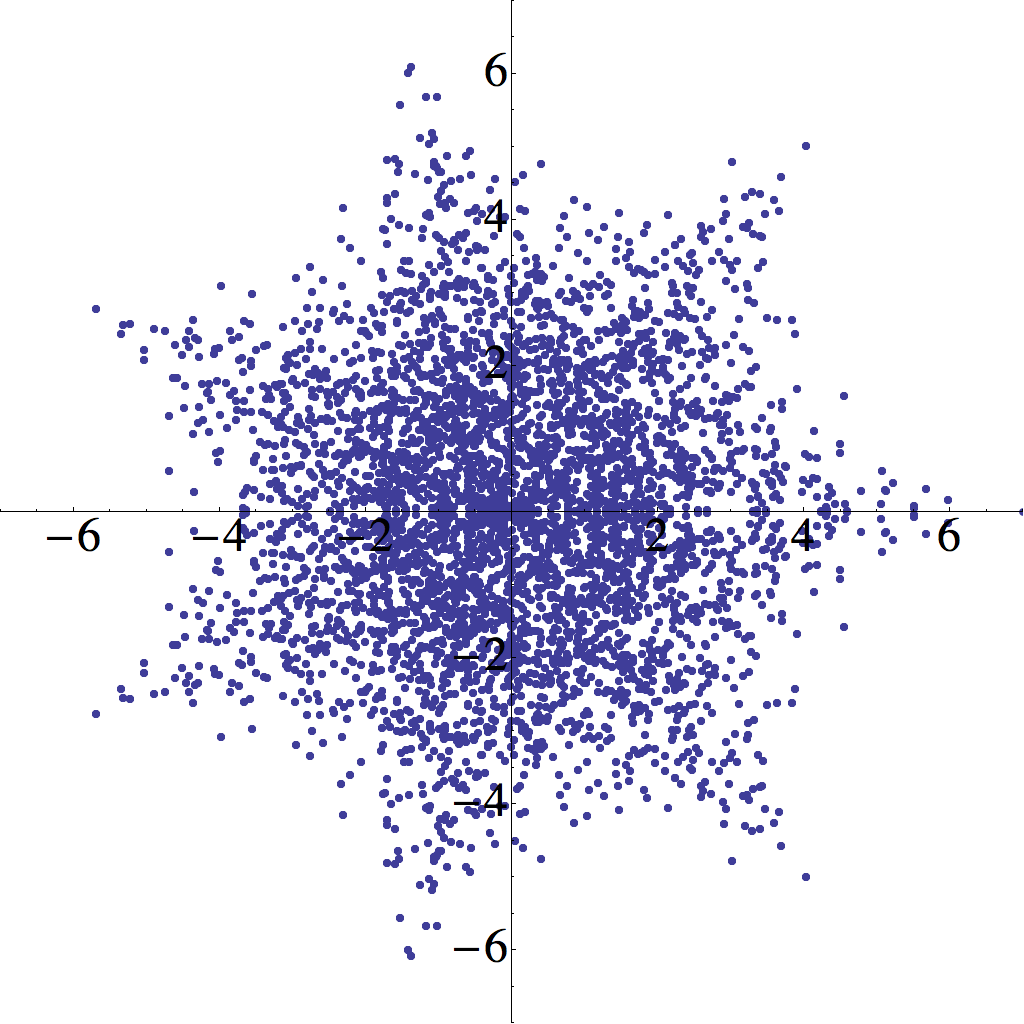}
	                \caption{{\scriptsize $p=32173$, $A=\langle 3223\rangle$}}
	        \end{subfigure}
	        \quad
	        \begin{subfigure}[b]{0.30\textwidth}
	                \centering
	                \includegraphics[width=\textwidth]{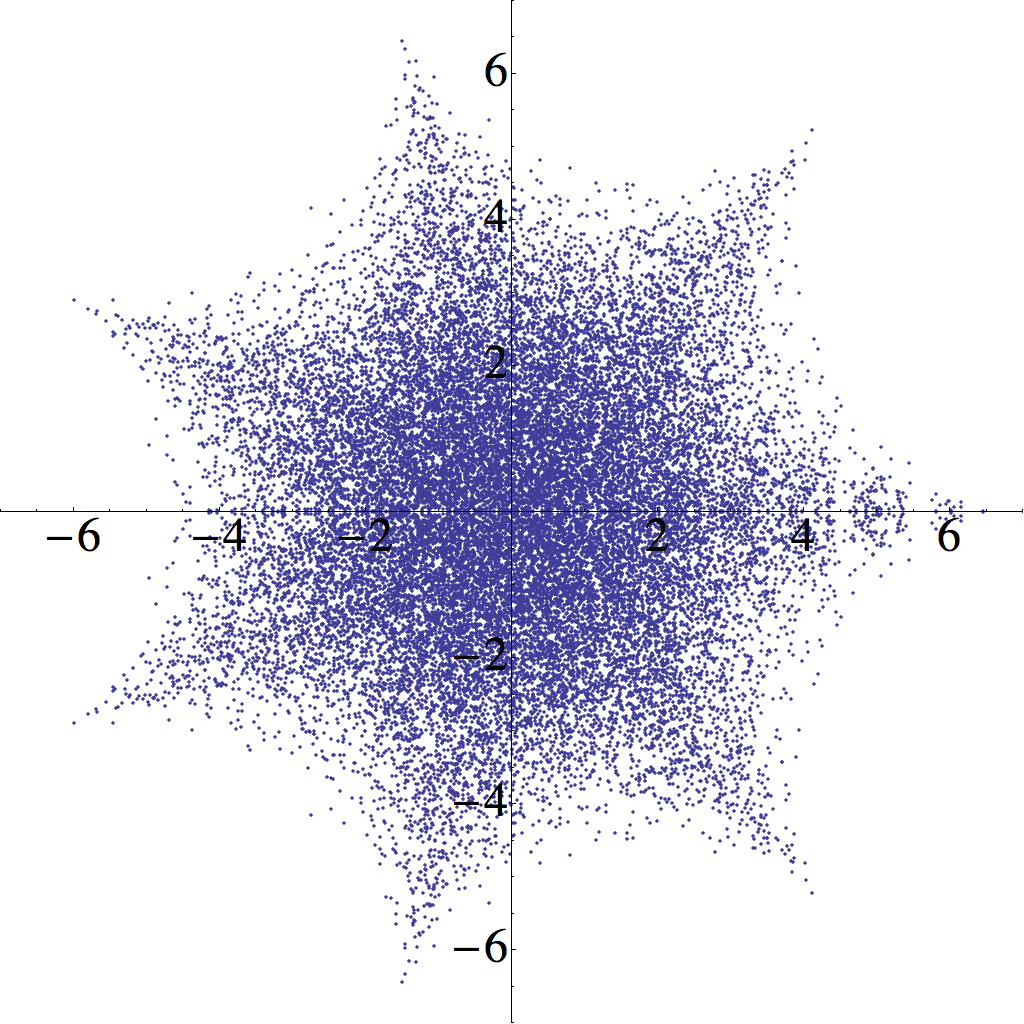}
	                \caption{{\scriptsize $p=200033$, $A=\langle 11073\rangle$}}
	        \end{subfigure}
	        \caption{Cyclic supercharacters $\sigma_X$ of $\Z/q\Z$, where $X=A1$, whose graphs fill out $|X|$-hypocycloids.}
		\label{fig:hypo}
	\end{figure}
	
	\begin{proof}[Pf.~of Thm.~\ref{TheoremHypo}]
		Computing the coefficients $b_{k,j}$ from \eqref{bkj} we find that
		$b_{k,j}=\delta_{kj}$ for $k=0,1,\ldots,d-2$, and $b_{d-1,j}=-1$ for all $j$, from which
		\eqref{eq-Mapping} yields
		\begin{equation*}
			g(z_1,z_2,\ldots,z_{d-1}) = z_1 + z_2 + \ldots + z_{d-1} + \frac{1}{z_1z_2\cdots z_{d-1}}.
		\end{equation*}
		The image of the function $g:\T^{d-1}\to\C$ defined above is the filled hypocycloid corresponding
		to the parameters $\kappa = d$ and $\lambda = 1$, as observed in \cite[$\S$3]{kaiser}.
	\end{proof}

\bibliography{GNGP}
\bibliographystyle{plain}
\end{document}